\documentclass[11pt]{amsart} % add ",oneside" after "12pt" for single-sided version
\usepackage{amsthm,amsbsy,amsmath,amssymb,amscd,amsfonts,array,mathrsfs,verbatim,enumerate,enumitem,xypic,color}
\usepackage{geometry,graphicx}

\usepackage[all,cmtip]{xy}
\usepackage{subcaption}
\xyoption{arc} 

\usepackage[utf8]{inputenc}

\newtheorem{thm}{Theorem}[section]
\newtheorem{prop}[thm]{Proposition}     
\newtheorem{lemma}[thm]{Lemma}
\newtheorem{cor}[thm]{Corollary}

\theoremstyle{definition}

\newtheorem{defn}[thm]{Definition}

\newtheorem{remark}[thm]{Remark}

\numberwithin{equation}{section}

\newcommand{\R}{\mathbb{R}}

\newcommand{\C}{\mathbb{C}}
\newcommand{\Z}{\mathbb{Z}}
\newcommand{\cl}{\mbox{cl}}

\newcommand{\ba}{\begin{array}}
\newcommand{\ea}{\end{array}}
\newcommand{\ra}{\rightarrow}

\newcommand{\st}{\textnormal{st}}

\newcommand{\tb}{{\bf tb}}
\newcommand{\tw}{\mbox{tw}}

\newcommand{\Fix}{\mbox{Fix}}
\newcommand{\id}{\textnormal{id}}
\newcommand{\conj}{\textnormal{conj}}
\newcommand{\smx}{s_{\textnormal{max}}}

\newcommand{\sg}{\Sigma_g}
\newcommand{\cst}{c_{\st}}
\newcommand{\mL}{\mathcal{L}}
\newcommand{\s}{S^1\times S^2}
\theoremstyle{plain}

\newcommand{\be}{\begin{eqnarray*}}
\newcommand{\ee}{\end{eqnarray*}}
\newcommand{\bne}[1]{\begin{eqnarray} \label{#1} }
\newcommand{\ene}{\end{eqnarray}}

\begin{document}

% Title information:
\title{Every real 3-manifold is real contact}

\author{Merve Ceng\.{ı}z}
\address{}
\email{ merveseyhun@gmail.com}

\author{Fer\.{ı}t \"{O}zt\"{u}rk}
\address{Bo\u{g}az\.{ı}\c{c}\.{ı} \"{U}n\.{ı}vers\.{ı}tes\.{ı}, Department of Mathematics, \.Istanbul, Turkey}
\email{ferit.ozturk@boun.edu.tr}

\subjclass[2010]{Primary 57K33, 57M60; Secondary 57M50, 53D35}

\begin{abstract}
A real 3-manifold is a smooth 3-manifold together with an orientation preserving smooth involution, which is called a real structure. A real contact 3-manifold is a real 3-manifold with a contact distribution that is antisymmetric with respect to the real structure. We show that  every real 3-manifold can be obtained via  surgery  along invariant knots starting from the standard real $S^3$ and  that this operation can be performed in the contact setting too. Using this result we prove that any real 3-manifold admits a real contact structure. As a corollary we show that any oriented overtwisted contact structure 
on an integer homology real 3-sphere can be isotoped to be real. 
Finally we give construction examples on $S^1\times S^2$ and lens spaces.
For instance on every lens space there exists a unique real structure that acts on each Heegaard torus as hyperellipic involution. We show that any tight contact structure on any lens space is real with respect to that real structure.
\end{abstract}
%\keywords{Real structures}

\maketitle

\section{Introduction}

A {\em real structure} on a smooth oriented $2k$- (respectively $(2k-1)$-) dimensional manifold 
is a smooth involution which is orientation preserving if $k$ is even and 
orientation reversing if $k$ is odd, with its fixed point set having dimension $k$ (respectively $k-1$), if not empty. 
The idea behind this definition is to mimic the complex conjugation on a complex analytic variety given by  functions with real coefficients. For $c_{M}$ a real structure on $M$, we call the pair $(M,c_M)$ a {\em real manifold} and the fixed point set of  $c_M$ the {\em real part}, denoted below by $\Fix~(c_M)$.

In this work we will heavily be interested in the 3-dimensional case:  a real structure on a smooth, closed, oriented 3-manifold is an orientation preserving smooth involution with the fixed point set being either empty or 1-dimensional.  The standard example for a real 3-manifold is the 3-sphere $S^3\subset\C^2$ with the standard real structure $\cst=\conj|_{S^3}$ where $\conj$ denotes the complex conjugation on $\C^2$. This involution is known to be the unique real structure with nonempty real part on $S^3$ up to equivariant isotopy. This fact is a result of the culminated work on masse on the resolution of the Smith conjecture that states originally that a finite-order diffeomorphism of $S^3$ cannot have a knotted 1-dimensional fixed point set. See \cite{Wa} for a PL topological solution of that conjecture for even periods  and \cite{MoBa} for a detailed exposition of the generalizations and the related work.

Now let $\xi$ be an oriented contact structure on a smooth, compact, oriented real $(M,c_M)$. If $(c_M)_*(\xi)=-\xi$,  then $\xi$ is said to be $c_M$-real and the triple $(M,c_M,\xi)$ is called a $c_M$-real  contact  manifold (see \cite{OzSa1}, \cite{OzSa2} for definitions and discussion). The obvious example is  $S^3\subset\C^2$ with the real structure $\cst$ and the unique tight contact structure $\xi_{\st}$ on $S^3$. More generally real contact 3-manifolds appear naturally as link manifolds of isolated complex analytic singularities defined by analytic functions with real coefficients: the natural tight contact structure induced on the link manifold by the complex tangencies is real with respect to the real structure determined by the complex conjugation. Even more generally, there are various basic conditions for a hypersurface in a real symplectic manifold (i.e. a real smooth manifold with an antisymmetric symplectic form) which makes it naturally a real contact manifold (see e.g. \cite{Ev}, notably 
for the observation there in Proposition~1.2.4, and for more examples.)

The positive contact structures on a closed, oriented 3-manifold are associated with  
the open book decompositions on the manifold via the Giroux correspondence; indeed open books and contact structures  are in one-to-one correspondence up to positive stabilizations and contact isotopy respectively \cite{Gi}.  Similarly on a real 3-manifold one can introduce the notion of a real open book decomposition. In \cite{OzSa2}  we have taken several first steps  towards a Giroux correspondence between  the real open books and the real contact structures.
Namely it has been proven there that every real open book supports a real contact structure, that every real contact structure is supported by a real open book, and that two real contact structures supported by the same real open book are equivariantly isotopic (see \cite{OzSa2} for definitions and the exact statements).

Nevertheless it remained an open question whether real open books and real contact structures always exist on a given  real 3-manifold, for example whether every real 3-manifold has a contact structure  at all that is real with respect to the given real structure. This last question, which 
was raised in  \cite{OzSa2} and \cite{Ev}, can be considered as the real version of J.~Martinet's result on the existence of contact structures on closed 3-manifolds \cite{Ma}. The main purpose of the present work is to answer that question affirmatively:

\begin{thm}
Every real 3-manifold admits a real contact structure.
\label{emartinet}
\end{thm}

One of the standard ways to prove Martinet's theorem is to recall that every closed 
3-manifold can be obtained via a surgery on a link in $S^3$ and then to argue that this surgery can be performed in the contact setting. A usual expression of the former fact is through the  well-known Lickorish-Wallace theorem, which states that every closed orientable 3-manifold may be obtained by a surgery along a link  in $S^3$ where each Dehn surgery coefficient is an integer  (see e.g. \cite{Ro}). 

In order to follow this track in the equivariant contact setting,  we first define and investigate in Section~\ref{EQS} the notion of Dehn surgery in the equivariant and contact equivariant setup in real contact 3-manifolds. Among others we detect
explicitly when equivariant contact $(1/l)$-surgery ($l\in\Z$) along equivariant Legendrian knots is possible (Theorem~\ref{contactsurgery}). One of the direct corollaries of that discussion is the following which is proven in Section~\ref{EQS}.

\begin{prop}
Any overtwisted contact structure on $S^3$ can be isotoped to be $\cst$-real. More precisely any overtwisted contact structure on $S^3$ can be obtained by an equivariant contact surgery in $(S^3,\cst,\xi_{\st})$.
\label{OTinS3}
\end{prop}

Section~\ref{S3R3M}, which is not directly related to the proof of Theorem~\ref{emartinet} and may be skipped in the first reading, 
suggests a method to obtain a given real 3-manifold from the standard real $S^3$ through a sequence of single surgery operations and intermediate real  3-manifolds,
in a way that each next surgery is equivariant in the previous 
intermediate real 3-manifold. We call such a link --constituted of an ordered collection of knots-- {\em recursively invariant}.
Thus  we prove the following theorem  in Section~\ref{S3R3M}, which can be considered as a recursively equivariant version of the Lickorish-Wallace theorem. (See Theorem~\ref{arda} for the detailed, precise version.)
\begin{thm} 
Every closed real 3-manifold can be obtained via a finite number of Dehn surgeries  along an ordered, recursively invariant collection of  knots starting from the real 3-sphere $(S^3,\cst)$.
\label{interm}
\end{thm}
Due to the recursive nature of the construction, this theorem does not lend itself easily to applications, for example in obtaining surgery diagrams in $S^3$. 
Instead, in Section~\ref{ELWT} 
we prove an equivariant version of the Lickorish-Wallace theorem. 
\begin{thm}[Equivariant Lickorish-Wallace Theorem]
Every closed, oriented real 3-manifold can be obtained via equivariant Dehn surgery along an equivariant link $\mL$ in the real 3-sphere $(S^3,\cst)$.
The equivariant link can be taken as $\mL= L\cup L_S\cup \overline{L}$ where $L_S$ is a $\cst$-equivariant unlink, $\cst(L)=\overline{L}$ and all the surgery coefficients can be taken as $\pm 1$,
(with respect to a framing induced by an invariant Heegaard surface).
\label{elw}
\end{thm}
The proofs of Theorems~\ref{interm}~and~\ref{elw} respectively follow the proof of the Lickorish-Wallace Theorem where we start with suitable decompositions for $S^3$ and the given 3-manifold, and look for appropriate factorizations of diffeomorphisms on  Heegaard surfaces in the mapping class group (recursively invariant factorization and equivariant factorization respectively).

Employing Theorem~\ref{elw}, we prove our main result, Theorem~\ref{emartinet}, in Section~\ref{3MRCS}.
For the proof it suffices to show that  the equivariant link in Theorem~\ref{elw} can be chosen appropriately so that it is possible to turn the equivariant surgeries into equivariant contact $(\pm1)$-surgeries. The proof is constructive and  produces an explicit algorithm that allows explicit equivariant contact surgery descriptions for real contact 3-manifolds.

An immediate consequence of Theorem~\ref{emartinet} is the following, which we prove in Section~\ref{3MRCS}.
\begin{cor}
Any oriented overtwisted contact structure on an integer homology real sphere $(\Sigma,s)$ of dimension~3 can be isotoped to be $s$-real.
\label{OTinZhomo}
\end{cor}

In Section~\ref{EXs} we produce examples on $S^1\times S^2$ and lens spaces.  In the first part we show 
\begin{thm}
The unique tight contact structure on $S^1\times S^2$ is real with respect to three of the four possible real structures on $S^1\times S^2$.
\label{3ofs1xs2}
\end{thm}
We do not know the answer for the last real structure.

Finally on any lens space there exists a unique real structure that acts on each Heegaard torus as hyperelliptic involution (coined as type~A in  \cite{HoRu}). 
We prove in Section~\ref{lenstypeA}

\begin{thm}
	For any $p>q>0$, every tight contact structure on $L(p,q)$ is $A$-real.
\label{A-real}
\end{thm}

\textbf{Acknowledgements.} 
The authors thank the anonymous referee for essential corrections and suggestions;  and the Mathematics Village, Izmir, for their hospitality during a research-in-pairs stay.
The second author is grateful to Sinem Onaran and Marc Kegel for their comments on contact diagrams.
\newpage

\section{Preliminaries and equivariant contact surgery}
\label{EQS}

\subsection{Basic definitions}
\label{basic}
In the sequel we always reside in the smooth category, both for spaces and maps, for the sake of keeping the rapport between involutions and contact structures. We note that 
in the topological category some claims in the equivariant realm may fail, e.g. the Smith conjecture \cite{MoZi}.

On a solid torus $S^1\times D^2$, there are four real structures up to isotopy through real structures  \cite{Ha}.
We choose  an oriented  identification of the boundary $T^2$ with $\R^2_{(x,y)}/ \Z^2$, where $x$ 
direction corresponds to the meridional direction of $T^2$, and fix the coordinates of 
$S^1\times D^2$  as  $(y, t, x)$ where $t$ is the radial direction of $D^2$.
In these coordinates, the four real structures on the solid torus are:

\noindent (1) $c_1:(y,(x,t))\mapsto (-y,(t,-x))$;\\
(2) $c_2:(y,(t,x))\mapsto (y,(t,x+\frac{1}{2}))$;\\
(3) $c_3:(y,(t,x))\mapsto (y+\frac{1}{2},(t,x))$;\\
(4) $c_4:(y,(t,x))\mapsto (y+\frac{1}{2},(t,x+\frac{1}{2}))$.

Any orientation
preserving involution on $T^2$ can be extended to an involution on $S^1\times D^2$. Such an extension is
unique up to isotopy and fixes a core of the solid torus setwise \cite{Ha}. A $c_j$-knot is by definition a knot which has a $c_j$-equivariant neighborhood.

Likewise there are three involutions on the core circle 
up to isotopy through involutions: reflection, identity and rotation by $\pi$ (antipodal map). 
We also denote them by $c_1$, $c_2$ (or id) and $c_3$ respectively. Note that $c_4|_{\textnormal{core}}=c_3$ too.

An embedded Heegaard decomposition for a real 3-manifold $(M,c_M)$ is said to be {\em real} if $c_M$ exchanges the two Heegaard handlebodies. In this case $c_M$ restricts to a real structure on the Heegaard surface $H$, reversing the orientation. 
It was proven by T.~A.~Nagase that every real 3-manifold admits
an (embedded) real Heegaard decomposition \cite[Proposition~2.4]{Na}.
As usual one can also construct a real 3-manifold by a pair of  handlebodies, both diffeomorphic to a handlebody $U$, and a gluing map $c:U\ra U$  which is a real structure on $U$ \cite{OzSa2}. Such a  decomposition of the 3-manifold is called an abstract real Heegaard decomposition, denoted in the sequel by the pair $(\partial~U,c)$.
The minimal genus among all Heegaard surfaces of all possible real Heegaard decompositions of $(M,c_M)$
is called the real Heegaard genus of $(M,c_M)$. It is greater than or equal to the Heegaard genus of $M$. See \cite{OzSa2} for a detailed discussion on real Heegaard decompositions.

A particular way to produce a real Heegaard decomposition is through real open books \cite{OzSa2}.
Let $(S, f)$ be an abstract open book, where S is a compact surface with boundary and $f : S \ra S$
is the monodromy with $f|_{\partial S}=$~identity. For a real structure $c$ on $S$,  the triple $(S, f, c)$ with $f \circ c = c \circ f^{-1}$ is called an (abstract) real open book.
The map $c_{\pi}=c\circ f$ is a real structure on the page~$\pi$ of the open book so that $f=c\circ c_{\pi}$. 
An abstract real open book determines a real 3-manifold $(M,c_M)$ uniquely and canonically. The union of the page~0 and the page~$\pi$ is a real Heegaard surface in 
$M$ and the real structure $c_M$ is the identity map between the two identical handlebodies of the Heegaard splitting. 
The restrictions of $c_M$ to the page~0 and the page~$\pi$ are respectively $c$ and $c_{\pi}$. As in the usual setting, there is a notion of positive real stabilization developed in  \cite{OzSa2}. Up to equivariant isotopy there are 9 distinct ways to attach handles to $S$ and to extend the real structure over the new handles (see  \cite[Figure~3]{OzSa2}).

In the sequel, instead of using the term $c_M$-real, we usually drop the reference to $c_M$ whenever the real structure is understood. The real structures and the real parts will be in red wherever color is possible. We assume that all contact structures are oriented and positive.

\subsection{Equivariant surgery}

Let $(M, c_M )$ be a closed, oriented real 3-manifold and $K$ be a $c_M$-invariant knot. Then $K$ has a
unique equivariant tubular neighborhood $N (K)$, which is equivariantly
isotopic to one of the real solid tori $(S^1 \times D^2 , c_i)$.
An invariant knot $K$ is called a $c_i$-knot if it has an equivariant neighborhood of type $(S^1 \times D^2 , c_i)$. The common name used for a $c_1$-knot is a strongly invertible knot (under the involution).

A topological $(p/q)$-surgery ($p,q\in\Z$) 
along the knot $K$ takes the meridian-longitude pair $(\mu_0,\lambda_0)$ of the new solid torus $N_0$ to
$(p\mu+q\lambda,p' \mu+q' \lambda)$ on the boundary of the excised neighborhood $N$ via a gluing map $\varphi$. Note that here $pq'-qp'=-1$ and $\lambda$ must be chosen. (The following discussion is with respect to a fixed $\lambda$.) This surgery can be performed naturally in the real setting along a $c_i$-knot.  The real extension $c_j$ over the surgered solid torus is unique in the following way.
Indeed in every possible case it suffices to check that $c_i\circ \varphi = \varphi \circ c_j$.
\begin{itemize}
\item If $K$ is a $c_1$-knot then the real structure extends as $c_1$.

\item If $K$ is a $c_2$-knot then the real structure extends as 
$\left\{ \ba{l} c_2 \mbox{ if $q$ even;} \\ 
 c_3 \mbox{ if $q$ odd, $q'$ even;} \\ 
 c_4 \mbox{ if $q$ odd, $q'$ odd;} \ea \right.$

\item If $K$ is a $c_3$-knot then the real structure extends as 
$\left\{ \ba{l} c_2 \mbox{ if $p$ even;} \\ 
 c_3 \mbox{ if $p$ odd, $p'$ even;} \\ 
 c_4 \mbox{ if $p$ odd, $p'$ odd;} \ea \right.$

\item If $K$ is a $c_4$-knot then the real structure extends as 
$\left\{ \ba{l} c_2 \mbox{ if $p+q$ even;} \\ 
 c_3 \mbox{ if $p+q$ odd, $p'+q'$ even;} \\ 
 c_4 \mbox{ if $p+q$ odd, $p'+q'$ odd.} \ea \right.$
\end{itemize}

\begin{defn}
If a $c_I$-solid torus is excised and a $c_J$-solid torus is glued back, we call such a surgery of type~$I_J$. As a final type of equivariant surgery, consider a knot $K$ satisfying $K \cap  \Fix~ (c_M ) = \emptyset$ and its disjoint copy $K' = c_M (K)$. An equivariant pair of surgeries performed along $K$ and  $K'$ will be coined as a type-5 surgery (along $K$ and  $K'$).
Since $c_M$ is an orientation preserving homeomorphism, the surgery framings along $K$ and  $K'$ are equal.
\end{defn}

\begin{remark}
Note that a $2_3$-surgery followed by a Dehn twist along the meridian (which alters the parity of $q'$) produces a manifold equivariantly diffeomorphic to one obtained by a single $2_4$-surgery. Similarly $3_3$- and $3_4$-surgeries are equivalent in that sense since a meridional Dehn twist alters the parity of $p'$. Similarly for $4_3$- and $4_4$-surgeries. We will use this subtle remark repeatedly in the sequel since we are content with diffeomorphisms, not isotopies. However if the latter is in question then one should a priori distinguish between  $I_3$- and $I_4$-surgeries, $I=2,3$ or $4$.
\label{izotopdegil}
\end{remark}

Let us also note that in case of a $1_1$ surgery, the number $|\Fix|$ of connected
components of the real part  may change. In fact, for a $c_1$-knot $K$,  if the fixed points of $K$ belong to the same (respectively different) component(s) of the real part, $|\Fix|$ either increases (respectively decreases) by 1 or stays the same.
Meanwhile,  in case of $2_3$- or $2_4$-surgery, $K$ is a $c_2$-knot (a real knot) and $|\Fix|$ decreases
by 1; in case of $3_2$- or $4_2$-surgery, $|\Fix|$ increases by 1.
The $3_4$- and $4_3$-surgeries do not alter  $|\Fix|$.

\subsection{Equivariant contact surgery}
The possibility of contact surgery for a rational coefficient relies on the  existence of a tight contact solid torus with the required contact structure on its convex boundary. It is known that the germ of a contact structure near the convex boundary is determined by a collection of curves on the boundary, called the dividing set. On a convex torus, this picture can be standardized further to obtain 
linear curves as the dividing set, so that the common slope of these curves determine
the contact structure near the torus. Such convex tori are said to be in standard form. 
The slope on the standard contact neigborhood of a Legendrian curve $L$ is determined by the contact twisting (denoted $tw(L)$) of the curve. The twisting is well-defined after a choice of a longitude for $L$.   
(See e.g. \cite{Ge} or \cite{Ho} for a thorough discussion for convex surfaces, slopes and twisting.)

Similarly the possibility of equivariant contact surgery for a nonzero rational coefficient relies of course on the  existence of an equivariant tight contact solid torus with the required slope on its convex boundary. The equivariant counterpart of the terms above (i.e. the standard equivariant contact neigborhood theorems, equivariant convex surfaces etc.) has been studied in  \cite{OzSa1} and \cite{OnOz}. Here we put together the previously known existence results for equivariant tight solid tori after \cite{OnOz}.

\begin{thm}
We have the following listed existence/nonexistence results regarding equivariant tight solid tori with convex boundary. In case of existence, it is unique up to  equivariant  contact isotopy relative to the equivariant convex boundary in standard form such that the dividing set has slope $s$ and has exactly two connected components. Below $k\in\Z$; $1/0$ is considered as $\infty$.
\begin{itemize}
	\item A $c_1$-real tight solid torus with $s=1/k$ exists.
	\item A $c_2$-real tight solid torus exists if and only if $s=1/k$.
	\item A $c_3$-real tight solid torus exists if and only if $s=1/(2k+1)$.
	\item A $c_4$-real tight solid torus exists if and only if $s=1/(2k)$.
	\end{itemize}
All of the solid tori above are the standard neighborhoods of equivariant Legendrian knots with $tw=1/s$. Here $tw$ is with respect to a fixed longitude.
\label{realst}
\end{thm}

Now let $K$ be an equivariant Legendrian knot in a real tight contact manifold $(M,\xi,c_M)$. Let $N$ be a standard equivariant contact neighborhood of $K$,  the existence of which is warranted by \cite{OzSa1} and \cite{OnOz}.
A contact $(p/q)$-surgery on $K$ is with respect to the choice of the longitude $\lambda$ as the dividing set on the boundary of $N$.
To have a well defined contact structure on the surgered 3-manifold, the dividing sets of $\partial N_0$ and $\partial N$ must match. By the identification above, the curve     $-p'\mu_0+p\lambda_0$ maps to the longitude $\lambda$, so the solid torus $N_0$ should be a tight contact solid torus with boundary slope $-p/{p'}$.

We now repeat this discussion in the equivariant setting. We will be  interested in the following cases.

\textbf{Case 1.} $K$ is a $c_1$-knot. The extension would be $c_1$. Since tight $c_1$-solid tori of slope $(1/p')$ ($p'\in\Z$) uniquely exist, equivariant contact  $(1/q)$-surgery of type~$1_1$ is uniquely defined up to equivariant contact isotopy. We fill in the standard $c_1$-real tight neighborhood of a $c_1$-invariant knot with $tw = -p'$ .

\textbf{Case 2.}  $K$ is a $c_2$-knot. It follows from Theorem~\ref{realst} that equivariant contact  $(p/q)-$surgery is defined if and only if $p=1$.  Similarly for the other cases below.

If $q$ and $q'$ are both odd, then $p'$ must be even and the real contact structure extends uniquely  as a $c_4$-tight solid torus with slope $-1/{p'}$ ($\infty$ included here and below whenever $p'=0$).

If $q$ is odd and $q'$ is even, then $p'$ must be odd and the real contact structure extends uniquely as a $c_3$-tight solid torus with slope $-1/{p'}$.

If $q$ is even, then the real contact structure extends as a $c_2$-tight solid torus with  slope $-1/{p'}$.

Note for the first two cases here that the parity of $q'$ is not well-defined: given $p$ and $q$ one can alter the parity of $p'$ and $q'$ by meridional Dehn twists. This is not a problem in the contact setting since a meridional Dehn twist extends over a solid torus and is smoothly isotopic to the identity. However it is not isotopic to the identity through real structures (See Remark~\ref{izotopdegil} and the conclusion in Theorem~\ref{contactsurgery}). 

\textbf{Case 3.} $K$ is a $c_3$-knot. Then $p=1$ as before. 

If $p'$ is odd, then the real structure extends as $c_4$ but no $c_4$-real tight solid torus exists with slope $-1/{p'}$ for odd $p'$.

If $p'$ is even, then the real structure extends as $c_3$ but no $c_3$-real tight solid torus exists with slope $-1/{p'}$ for even $p'$.

\textbf{Case 4.} $K$ is a $c_4$-knot. Then $p=1$. 

If $p+q=1+q$ is odd (i.e. $q$ is even; so $q'$ is odd) and $p'+q'$ is odd, then $p'$ must be even and the real contact structure extends uniquely as a $c_4$-tight solid torus with slope $-1/{p'}$.

If $q$ is even and $p'+q'$ is even, then $p'$ must be odd and the real contact structure extends uniquely as a $c_3$-tight solid torus with slope $-1/{p'}$.

If $q$ is odd, then the real contact structure extends uniquely as a $c_2$-tight solid torus with slope $-1/{p'}$.

Similar remark as the one following Case~2 above applies here for the  parity of $p'$: it is not well-defined and can be altered. 
Let us summarize the above discussion.

\begin{thm}[Equivariant Contact Surgery along Equivariant Legendrian Knots]
Equivariant contact surgery along a Legendrian $c_3$-knot is impossible. Equivariant contact $(1/q)$-surgery ($q\in \Z$) along a Legendrian $c_1$-, $c_2$- or $c_4$-knot is uniquely defined up to equivariant contactomorphism.
For a $c_2$- or $c_4$-knot the only possible contact surgery coefficient is  $1/q$. 

Moreover for the  surgery types $1_1,2_2$ and $4_2$ the  uniqueness is up to equivariant contact isotopy.

Table \ref{tab:extensions} details the cases for equivariant contact surgery.
The last column describes the type~$c$ of the glued back solid torus:  a standard $c$-real tight neighborhood of a $c$-invariant knot with $tw=-p'$. 
\label{contactsurgery}
\end{thm}
\begin{table}[ht]
	\centering
	\begin{tabular}{|c| c| c|}
		\hline
		Knot &  Case & Glue back\\ [1ex] 
		\hline
		$c_1$  &   & $c=c_1$ \\
		\hline
		$c_2$  & $q$ odd, $q'$ odd & $c=c_4$\\
		\hline
		  & $q$ odd, $q'$ even & $c=c_3$\\
		\hline
		 & $q$ even & $c=c_2$\\
		\hline
		$c_3$ & $\emptyset$  & \\
		\hline
		$c_4$ & $q$ even, $p'$ even & $c=c_4$\\
		\hline
		 &  $q$ even, $p'$ odd & $c=c_3$\\
		\hline
		 & $q$ odd & $c=c_2$\\
		\hline
	\end{tabular}
\caption{Equivariant contact  $1/q$-surgery along $c_i$-knots. We glue back the standard $c$-real tight neighborhood of a $c$-invariant knot with $tw=-p'$}
	\label{tab:extensions}
\end{table}

\noindent {\it Proof of Proposition~\ref{OTinS3}.}
The number of  overtwisted contact structures on $S^3$ up to contact isotopy  is countably infinite. They are distinguished by  the $d_3$~invariant, which is a half integer for $S^3$. (For $d_3$, see \cite{Go} where it was first defined or see e.g. \cite{DiGeSt}.) The overtwisted structure $\xi_{-2p}$ with $d_3=-2p+1/2$, $p\in \Z$,  can be obtained by a contact surgery given explicitly in, for example, \cite[Figure~8]{EtOz}. Observing the 
symmetry here, one can immediately turn this diagram into an equivariant contact one along a pair of $c_1$-real knots with surgery coefficients $+1$ and $-1/p$ 
(see Figure~\ref{xipinS3}), except for $p=0$ when there is a single knot, the one on the left  in Figure~\ref{xipinS3} with surgery coefficient $+1$. 
These equivariant contact surgeries are possible by Theorem~\ref{contactsurgery}. 
Recall that the equivariant contact connected sum between two real contact 3-manifolds is well-defined thanks to the fact that there is a unique real tight 3-ball \cite{OzSa1}. 
Thus taking connected sums of arbitrarily many $(S^3,\cst,\xi_{k})$ with $k=0$ and $d_3=1/2$ (respectively $k=-2$ and $d_3=-3/2$) 
one obtains every  overtwisted 3-sphere with positive (respectively negative) $d_3$ (cf.  \cite[Lemma~4.2]{DiGeSt}).
Alternatively one could consider the real contact 3-spheres obtained by taking the connected sum between $(S^3,\cst,\xi_{0})$ and the 3-sphere in Figure~\ref{xipinS3} for varying $p\in\Z$. Note that in each case the equivariant contact diagram is simply the disjoint union of the previous ones.

In general one must also check that the final real structure is the desired one. But here that is immediate by the uniqueness of  $\cst$. 
\hfill $\Box$

\begin{figure}[h]
\centering
\resizebox{7cm}{!}
{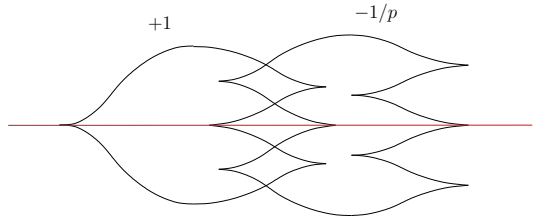}
\caption{The $\cst$-real overtwisted structure $\xi_{-2p}$ on $S^3$, $p\in\Z$. In the case $p=0$, the $\infty$-surgery on the right disappears.}
\label{xipinS3} 
\end{figure}

\section{Going recursively from real $S^3$ to any real 3-manifold}
\label{S3R3M}
In this section we will prove the following theorem of recursive nature.

\begin{thm} 
\label{arda}
Every closed real 3-manifold $(M,c_M)$ can be obtained via a finite number of Dehn surgeries  along an ordered, recursively invariant collection of  knots starting from the real 3-sphere $(S^3,c_{\st})$.

More precisely, there is a sequence $\{(M_j,s_j;K_j)\}_{j=0}^k$ of real 3-manifolds,  a common Heegaard surface $H$ 
and $s_j$-invariant knots $K_j\subset H \subset M_j$ such that $(M_0,s_0)=(S^3,c_{\st})$ and $(M_k,s_k)$ is equivariantly diffeomorphic to $(M,c_M)$ and each $(M_{j+1},s_{j+1})$, $0\leq j \leq k$, is obtained from $(M_j,s_j)$ via a $\pm 1$  surgery along $K_j$ where the framing along $K_j$ is determined by $H$; at each step, the real structure is canonically extended to the surgered region.
\end{thm}

First we will prove several lemmata regarding factorizations in the mapping class group of a real Heegaard surface. To start we need some preliminaries. Two real structures  (orientation reversing involutions) $r$ and $s$ on an closed, oriented genus $g$ surface $\sg$ are said to be equivalent if there is an orientation preserving diffeomorphism $h$ such that $r = h\circ s\circ h^{-1}$. 
It is well-known that the equivalence class of $r$ is determined by the number of connected components of the real part $\Fix~(r)$,  and connectedness of  $\sg -\Fix~(r)$. (This follows from considering the quotient surface and the classification of 2-manifolds.)
In case $\Fix~(r)$ is separating (i.e. $\sg -$ Fix~$(r)$ has exactly  2  connected components),   then $1\leq |\Fix~(r)| \leq g+1$ and $g$ and $|\Fix~(r)|$ have opposite parities. We denote by $\smx$ the maximal real structure in the standard form (see Figure~\ref{humph}). 
If $\Fix~(r)$ is non-separating (i.e. $\sg -\Fix~(r)$ is connected), then   $0\leq |\Fix~(r)| \leq g$. In this case, the real structure in the standard form is denoted by $\mathfrak{s}_{|\mbox{\footnotesize Fix}~(r)|}$.

\begin{figure}[t]
\centering
\resizebox{15cm}{!}
{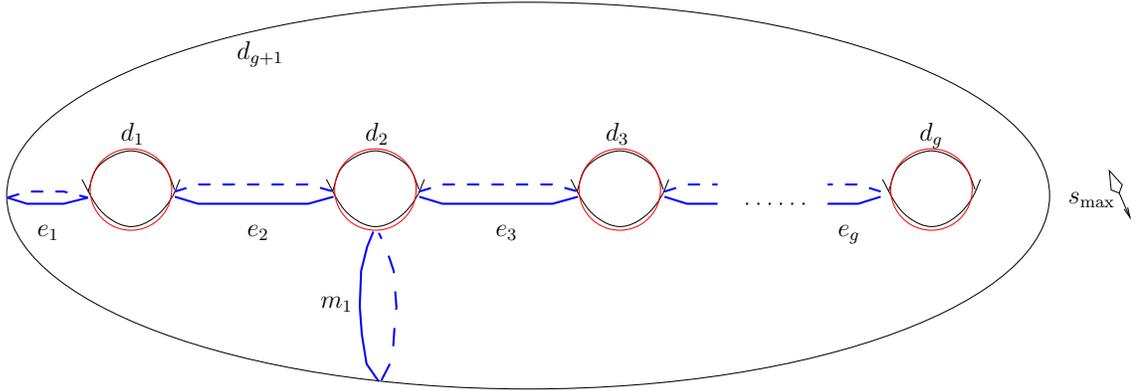}
\caption{The maximal real structure $\smx$ on $\sg$. The set of fixed points is $d_1\cup\ldots\cup d_g\cup d_{g+1}$; here $d_{g+1}$ is the large outer closed curve. The Humphries generators of the mapping class group are $d_1,\ldots, d_g,e_1,\ldots, e_g,m_1$; each is an $\smx$-invariant curve.}
\label{humph} 
\end{figure}

Below $\tau_a$ always denotes the positive Dehn twist along a curve $a$ on a surface.

\begin{lemma} 
Any real structure $s$ on $\sg$ with $|\Fix~(s)| = k$ can be expressed in the form
\begin{equation} s =\tau_{\beta}^m \cdot \tau_{b_u}\ldots \tau_{b_1}\cdot p \cdot \smx
\end{equation}
where $ p = \tau_{a_1}^{\sigma_1} \ldots \tau_{a_w}^{\sigma_w}  \cdot \tau_{a_w}^{\sigma_w}  \ldots \tau_{a_1}^{\sigma_1}$   is a
palindrome (possibly empty) with even length and with each $a_j$ $\smx$-invariant, $\sigma_j \in\Z$; each $b_j \subset \Fix~ (p \cdot \smx)$; 
and $\beta\subset \Fix~ (s)$ with $m = 1$ if $s$ is separating but not maximal and $m = 0$ otherwise.
\label{ifade}
\end{lemma}

\begin{proof} (i) Assume $s$ is non-separating. Set $u = g + 1 - k$ and take $\smx$-real circles $d_1,\ldots,d_u$ on $\sg$ . 
Then $\mathfrak{s}_k=\tau_{d_u}\ldots \tau_{d_1}\cdot \smx$ is the standard non-separating real structure with $k$ real circles
hence is conjugate to $s$. Then there is an orientation preserving diffeomorphism $f$ of $\sg$ so that
$$ \ba{ll}
s & = f \mathfrak{s}_k f^{-1} = f \cdot \tau_{d_u}\ldots \tau_{d_1}\cdot \smx \cdot f^{-1} \\
& = f \cdot \tau_{d_u}f^{-1} f \ldots f^{-1} f \tau_{d_1}\cdot f^{-1} f  \smx \cdot f^{-1} \\
& = \tau_{f(d_u)}\ldots \tau_{f(d_1)}\cdot f\smx f^{-1}.
\ea $$
Observe that since each $d_j$ is $\smx$-invariant, each $b_j=f(d_j)$ is $(f\smx f^{-1})$-invariant.  
Furthermore fixing a basis for the mapping class group of $\sg$ consisting of twists on $\smx$-invariant curves (e.g. the 
$\smx$-invariant Humphries generators in Figure~\ref{humph}), 
we can write
$$ f\smx f^{-1} 
= (\tau_{a_1}^{\sigma_1} \ldots \tau_{a_w}^{\sigma_w})\cdot \smx \cdot (\tau_{a_w}^{-\sigma_w} \ldots \tau_{a_1}^{-\sigma_1})
= \tau_{a_1}^{\sigma_1} \ldots \tau_{a_w}^{\sigma_w}\tau_{a_w}^{\sigma_w} \ldots \tau_{a_1}^{\sigma_1}\cdot \smx.$$

\noindent (ii) If $s$ is separating but not maximal, take any circle $\beta\subset \Fix~(s)$. Then $\tau_{\beta}^{-1}\cdot s$ is non-separating and the proof follows from part (i).

\noindent (iii) If $s$ is maximal then the equation holds with $m = 0$ and $u=0$ following  part (i).
\end{proof}

On a real surface $(\sg,s)$, we say that the product of Dehn twists $\tau^{\sigma_m}_{r_m}\cdots\tau^{\sigma_1}_{r_1}$ \textit{ satisfies the recursive $s$-invariance condition} (or is {\it recursively $s$-invariant}) if  for each $j=1,\ldots,m$, \\
(i) $\tau^{\sigma_{j-1}}_{r_{j-1}}\ldots\tau^{\sigma_1}_{r_1}\cdot s (r_j)= r_j$; \\
(ii) or $r_j$ and $r_{j+1}$ are disjoint, and $\tau^{\sigma_{j-1}}_{r_{j-1}}\ldots\tau^{\sigma_1}_{r_1}\cdot s (r_j)= r_{j+1}$. 

%In the latter case we use the notation $r_j, r'_j$ instead of $r_j, r_{j+1}$ to spare index.
We will set $s_0=s$ and, for $j>0$, $s_j=\tau^{\sigma_j}_{r_j} \cdot s_{j-1}$ in case (i).% and $s_j=\tau^{\sigma_j}_{r_j} \tau^{\sigma_j}_{r'_j} \cdot s_{j-1}$ in case (ii).

\begin{lemma} 
The even palindrome  $p$ in the previous lemma  has a  recursively $\smx$-invariant factorization in even powers.
\label{even}
\end{lemma}
\begin{proof}
First, using the fact that $p$ is an even palindrome,  observe that it can be written as
%(dropping the powers $\sigma_j$'s for the sake of clarity)
$$ p=(\tau^{\sigma_1}_{a_1}\cdots\tau^{\sigma_{w-1}}_{a_{w-1}}\tau^{\sigma_w}_{a_w}\tau^{-\sigma_{w-1}}_{a_{w-1}}\cdots\tau^{-\sigma_1}_{a_1})^2\cdots (\tau^{\sigma_1}_{a_1}\tau^{\sigma_2}_{a_2}\tau^{\sigma_3}_{a_3}\tau^{-\sigma_2}_{a_2}\tau^{-\sigma_1}_{a_1})^2\cdot  (\tau^{\sigma_1}_{a_1}\tau^{\sigma_2}_{a_2}\tau^{-\sigma_1}_{a_1})^2\cdot\tau^{2\sigma_1}_{a_1}.$$
Now for each $1\leq j<w$, set $f_j=\tau^{\sigma_1}_{a_1}\cdots\tau^{\sigma_j}_{a_j}$, $\bar{f_j}=\tau^{\sigma_j}_{a_j}\cdots\tau^{\sigma_1}_{a_1}$, and $r_1=a_1$, $r_{j+1}=f_{j}(a_{j+1})$. Then the factorization above becomes
$$ p=\tau^{2\sigma_w}_{r_w}\cdots \tau^{2\sigma_2}_{r_2}\tau^{2\sigma_1}_{r_1}.$$
We claim that this factorization is recursively $\smx$-invariant. In fact, $r_1=a_1$ is invariant under $\smx$; furthermore, for each $1\leq j<w$,
(dropping the powers $\sigma_j$'s for the sake of clarity)
\begin{eqnarray}
(\tau^2_{r_j}\cdots\tau^2_{r_1} \smx) (r_{j+1}) & = & f_j\bar{f_j} \smx (f_j(a_{j+1}))\nonumber \\
& = & f_j(\bar{f_j}\bar{f_j}^{-1}) \smx (a_{j+1}) \nonumber \\
& = & f_j(a_{j+1})=r_{j+1}. \nonumber
\end{eqnarray} 
To finish the proof, we should also note that $(\tau_{r_{j+1}}\tau^2_{r_j}\cdots\tau^2_{r_1} \smx) (r_{j+1}) = r_{j+1}$.
\end{proof}

The proof above  shows also that the composition of $\smx$ with $\tau_{r_j}$'s is always a real structure. In fact,  for every $1\leq j<w$, (without writing $\sigma_j$'s)
\begin{eqnarray}
(\tau^2_{r_j}\cdots\tau^2_{r_1} \smx) \cdot(\tau^2_{r_j}    \cdots \tau^2_{r_1} \smx)  & = &  (f_j\bar{f_j} \smx)\cdot( f_j\bar{f_j} \smx)\nonumber \\
& = & f_j \bar{f_j} \bar{f_j}^{-1} {f_j}^{-1}  \smx^2 =   \mbox{id} \nonumber
\end{eqnarray}
and similarly
$$
(\tau_{r_{j+1}}\tau^2_{r_j}\cdots\tau^2_{r_1} \smx) \cdot(\tau_{r_{j+1}}\tau^2_{r_j}\cdots\tau^2_{r_1} \smx)  = 
\tau_{r_{j+1}}\tau_{r_{j+1}}^{-1}  = \mbox{id}.
$$
Of course this fact is valid for every recursively $s$-invariant product $\tau_{r_w}^{\sigma_w}\ldots \tau_{r_1}^{\sigma_1}$. If we set $s_0 = s$
and, for $j > 0, s_j =\tau_{r_j}^{\sigma_j}\cdot s_{j-1}$, we see that $\{s_j\}_{j=0}^{j=m}$ is a sequence of real structures on $\sg$.

\noindent {\it Proof of Theorem~\ref{arda}.} 
We will follow the proof of the Lickorish-Wallace Theorem (see e.g. \cite{Ro}). Given the real 3-manifold $(M,c_M)$, consider a real Heegaard splitting of $M$. Let the Heegaard surface $H$ have genus $g$. Then we take a genus-$g$ nonseparating real splitting of $(S^3,c_{\st})$ 
(which exists for any $g>0$; see \cite{OzSa2} or \cite{Se}) %\marginpar{\tiny tezde eski bir lemmaya ref} 
so that the real map $s=c_{\st}|_H$ equals $\tau_{\delta_1}\ldots \tau_{\delta_g}\cdot \smx$ for some $\smx$-real disjoint curves $\delta_1,\ldots,\delta_g$.
Note that in these splittings, $s$ and $c=c_M|_H$   are gluing maps. In the usual proof of the Lickorish-Wallace Theorem, the composition $c\cdot s$ is expressed as a product of Dehn twists and then each twist is extended over  handlebodies. In our case, we express  $c\cdot s$  as a recursively $s$-invariant product and show that that factorization describes an appropriate recursively equivariant  sequence of $(\pm 1)$-surgeries. 
Now, using Lemma~\ref{ifade} and Lemma~\ref{even} we write:
$$ \ba{ll} c\cdot s & = (\tau_{\beta}^m  \tau_{b_u}\ldots \tau_{b_1}\cdot p \cdot \smx)\cdot (\tau_{\delta_1}\ldots \tau_{\delta_g}\cdot \smx) \\
& = \tau_{\beta}^m  \tau_{b_u}\ldots \tau_{b_1}\cdot p \cdot \tau_{\delta_1}^{-1}\ldots \tau_{\delta_g}^{-1}
\ea$$
where $p = \tau_{r_w}^{2\sigma_w}  \ldots \tau_{r_1}^{2\sigma_1}$ is a (possibly empty) recursively  $\smx$-invariant factorization with even powers,
$b_j$'s are in $\Fix~ (p \cdot \smx)$; $\delta_j$'s are $s$-invariant and $\smx$-real; and $\beta\in \Fix~(c)$ with $m = 1$ if $c$ is non-separating but not maximal and $m = 0$ otherwise.

We claim that the last factorization for $c\cdot s$ is recursively $s$-invariant. First since $\delta_j$'s are disjoint $s$-invariant, following the terms of Lemma~\ref{even} we have:
$$ \ba{l}
s_0=s, \; r_1=\delta_g; \\
s_1=\tau_{\delta_g}^{-1}\cdot s, \; r_2=\delta_{g-1}; \\
\vdots \\
s_{g}=\tau_{\delta_1}^{-1}\ldots \tau_{\delta_g}^{-1}\cdot       s = \smx, \; r_{g+1}=\delta_1. \
\ea $$
Next, $p$ is already recursively $\smx$-invariant.  Moving on with the remaining terms of the product we set  $s_{ g+w+1} = p \cdot \smx$ and $r_{g+w+2} = b_1$. Since $b_j$'s are disjoint and $(p\cdot\smx)$-real and $\beta$ is $c$-real the claim of recursive $s$-invariance follows.

Finally we claim that each of these steps determines
an appropriate equivariant $(\pm 1)$-surgery with respect to the framing determined by the initial surface $H$. As usual this is accomplished by {\em pushing} each curve equivariantly into the Heegaard handlebodies, as the following proposition shows, and thus  the proof of Theorem~\ref{arda} follows. 
\hfill $\Box$

\begin{prop}
Let $X$ and $Y$ be real 3-manifolds with real Heegaard splittings $(H,c)$ and $(H,s)$ respectively. Suppose $s=\tau^{\pm \sigma}_\alpha\cdot c$ where 
$\alpha$ is a $c$-invariant curve on $H$. Then a $(\mp 1/2\sigma)$-surgery in $X$ along $\alpha$, with  framing determined by $H$, followed by uniquely extending the real structure over the surgered region gives a real 3-manifold  equivariantly diffeomorphic to  $(Y,s)$.
\label{kapkun}
\end{prop}
\begin{proof}
Let $H_1$ and $H_2$ be the two handlebodies of the Heegaard splitting of  $(X,c)$
%We start with the following splitting of  $(X,c)$. 
and let $X_0\subset X$ be an equivariant neighborhood of $H$ diffeomorphic to $H\times [-1,1]$. Set $X_1=\cl(H_1 - X_0)$ which we consider identical to $X_2=\cl(H_2 - X_0)$. Then $X$ is equivariantly diffeomorphic to $X_1\cup_{\id|_{\partial X_1}} X_0\cup_{c|_{\partial X_0}}\! X_2$ (here the gluing maps are from the top boundary of the first space to the bottom boundary of the next) with the real structure $\tilde{c}$ defined as 
$$\tilde{c}:\left\{ \ba{ll} X_1\ra X_2, & x\mapsto x \\
%X_2\ra X_1, & x\mapsto x  \\
X_0\ra X_0, & x\mapsto c(x).
\ea \right.$$
Similarly, we consider such a splitting $(Y_1,Y_0,Y_2;\tilde{s})$ for $(Y,s)$. Here we consider $X_1,X_2,Y_1,Y_2$ identical and $X_0,Y_0$ identical. 

Let $\nu(\alpha)$ be an annulus neighborhood of $\alpha$ in $H$ and  $N=\nu(\alpha)\times [-1,1]$ be a  $c$-equivariant smooth neighborhood of $\alpha$ in $X_0$. Since $\alpha$ is both $c$- and $s$-invariant, with a slight abuse,  we will also denote by $N$ a small $s$-equivariant neighborhood of $\alpha$ in $Y_0$. We observe that the identity map (again with an abuse) is an equivariant homeomorphism from $X-N$ to $Y-N$, which can be made an equivariant diffeomorphism after smoothing.
Now,  with $T=S^1\times D^2$ and $\varphi$ a ($\mp 1/2\sigma$)-sloped diffeomorphism on $\partial N$, $(X-N)\cup_{\varphi} T$ is diffeomorphic to $Y$. Here the previous diffeomorphism  $\textnormal{id}_{X-N}$ extends trivially over $T$. 

What the essence in the claim of the proposition  is and what we have to show basically is that  the above extension can be performed equivariantly.
Indeed the real structure $\tau^{\pm 2\sigma}_{\alpha} c|_{\partial T}$ (here $\tau^{\pm 1}_{\alpha}$ is considered
to be a twist around a copy of $\alpha$ on $\partial T \cong_{\varphi} \partial N$) can be extended to a real structure over $T$ uniquely. In fact, \\
(i) If $c|_{\alpha}=c_1$ then   
%$\tau^{\pm 2\sigma}_{\alpha}  c|_{\partial T}$ will be extended over $T$. T
the unique extension over $T$  up to equivariant isotopy is $c_1$.  This corresponds to a type-$1_1$ equivariant surgery (see Theorem~\ref{contactsurgery}). \\
(ii) If $c|_{\alpha}=c_2$, i.e. $\alpha$ is real, then $\tau^{\pm 2\sigma}_{\alpha}  c|_{\partial T}$ is equivariantly diffeomorphic to $c_2$ since the surgery coefficient $q=\pm 2\sigma$ is even. This corresponds to a type-$2_2$ equivariant surgery. \\
(iii) If $c|_{\alpha}=c_3$ then $c|_{\partial N}= c_3$ or $c_4$ . The former is impossible on a real
Heegaard surface. Then $\tau^{\pm 2\sigma}_{\alpha}  c|_{\partial T}$ is equivariantly
isotopic to $c_3$ or $c_4$. This is an equivariant type-$4_4$ or $4_3$  surgery. Of course the choices here are equivariantly diffeomorphic.
\end{proof}

\section{Equivariant Lickorish-Wallace theorem}
\label{ELWT}

In this section, we will give an equivariant surgery description for a real 3-manifold and prove the  Lickorish-Wallace Theorem in the equivariant setting.  In general the surgery link  provided by the  Lickorish-Wallace Theorem need not accept any kind of symmetry in $(S^3,c_{\st})$. However we show here that it is always possible to construct an invariant surgery link. 

Consider a real three-manifold $(M,c_M)$ and a $c_M$-equivariant link $L=L_1\cup\ldots\cup L_n$ in $M$ decorated with an integer $n$-tuple $\sigma=(\sigma_1,\ldots \sigma_n)\in\Z^n$. Let $(M(L),c_M(L),\sigma)$ denote the real 3-manifold resulting from $c_M$-equivariant $\sigma$-surgery along $L$, i.e. the
collection of $\sigma_i$-surgeries along $L_i$'s, $1\leq i \leq n$. (As usual the surgery coefficients are with respect to a standard or given reference framing.) We will usually write $(M(L),c_M(L))$ instead, when the surgery coefficient $\sigma$ is clearly understood or is not explicitly required.

\begin{defn}
	On a real surface $(\sg,c)$, a product $\tau_{a_t}^{\sigma_t}\ldots\tau_{a_1}^{\sigma_1} \cdot\tau_{b_k}^{\pm1}\ldots\tau_{b_1}^{\pm1}\cdot \tau_{a_1'}^{\sigma_1}\ldots\tau_{a_t'}^{\sigma_t}$ of Dehn twists is called an \emph{equivariant product} for $c$ if  $c(a_i)=a_i'$ and $b_j$ are disjoint $c$-invariant curves for all $1\leq i\leq t$ and $1\leq j\leq k$ .
\end{defn}

We start with a technical lemma.

\begin{lemma}
	Let $(M,c_M)$ and $(M',c_{M'})$ be two real manifolds with the associated real Heegaard splittings $(H,c)$ and $(H,c')$ respectively. Assume that 
$c'c=\tau_{a_t}^{u_t}\ldots\tau_{a_1}^{u_1} \cdot\tau_{b_k}^{v_k}\ldots\tau_{b_1}^{v_1}\cdot \tau_{a_1'}^{u_1}\ldots\tau_{a_t'}^{u_t}$
with all $u_i,v_j\in\{-1,+1\}$ and that this factorization is an equivariant product for $c$. Then there is an equivariant link $L$ in $(M,c_M)$ decorated with $u_i$'s and $v_j$'s such that the real 3-manifold $(M(L),c_M(L))$ (with respect to a framing 
induced by a real Heegaard surface) is equivariantly diffeomorphic to the manifold $(M',c_{M'})$.
	\label{main}
\end{lemma}
\begin{proof}
Topologically, the manifold $M'$ can be obtained from $M$ following Lickorish and Wallace. Let $M=U_1\cup_c U_2$ and $M'=U_1\cup_{c'} U_2$ where $U_1$ and $U_2$ are two genus-$g$ handlebodies bounded by $H$. 
Instead, we consider a splitting of $(M,c_M)$, similar to the proof of Proposition~\ref{kapkun} but with more intermediate blocks, as follows. Take an equivariant neighborhood $H\times[-1,1]$  of the Heegaard surface $H$ with $c_M$ fibered in the sense that $c_M|_{H\times\{t\}}=c$. Keeping this model in mind, set  $H_i=H\times [i,i+1]$, $-t-1\leq i \leq t$. Then  $(M,c_M)$ is equivariantly diffeomorphic to the manifold  
$$U_1\cup_{\id} H_{-t-1}\cup_{\id}\ldots\cup_{\id} H_{-2}\cup_{\id} H\times[-1,1] \cup_{\id} H_{1}\cup_{\id} \ldots\cup_{\id} H_t \cup _c U_2$$
(as in the proof of Proposition~\ref{kapkun}, here and in every similar notation below  the gluing maps are to be understood from the {\em top} boundary of the previous space to the {\em lower} boundary of the next)
	with the real structure defined as$$\tilde{c}_{M}=
	\begin{cases}
	\id: U_1\to U_2, \\ %\hspace{3mm} x \mapsto x\\
%	\id: U_2\to U_1, \\ %\hspace{3mm} x \mapsto x\\
	c_M: H_i \to H_{-i-1}. % \hspace{3mm} x \mapsto c_M(x).
	\end{cases}$$ 
(See Figure~\ref{model}, left.) We will use this splitting and real structure for $(M,c_M)$, and write $c_M$ for  $\tilde{c}_{M}$, abusing the notation

We push the curves  $a_i$ to the knots ${K}_i$ on the surface $h_{-i}$ (which is the {\em lower} boundary of $H_{-i}$; see Figure~\ref{model}), the curves $a_i'$ to  $\overline{K}_i$ on  $h_{i}$ (lower boundary of $H_{i}$)  and the curves $b_j$ determine knots $C_j$ on $h_0$ (which are mutually disjoint by assumption), so that we get the surgery link 
$$(L,\sigma)= (\overline{K}_t,u_t)\cup\ldots \cup(\overline{K}_1,u_1)\cup (C_k,v_k)\cup\ldots\cup(C_1,v_1)\cup\ldots\cup(K_t,u_t)$$
for the manifold $M'$. The surgery coefficients here are with respect to the associated surface $h_i$. Note that $L$ is a ${c_M}$-equivariant link, as the knot pair $K_i$ and $\overline{K}_i$ are swapped by $\tilde{c}_{M}|_H={c_M}|_H=c$  for each $1\leq i\leq t$ and $C_j$ is a $c$-invariant knot for all $1\leq j\leq k$. 

We will prove in two steps that the resulting real manifold $(M(L),c_{M}(L))$ is not only  diffeomorphic but also equivariantly diffeomorphic  to $(M',c_{M'})$.
First we will show that $(M(L),c_M(L))$ is equivariantly diffeomorphic to  the intermediate real manifold $(E,e)$ where 
	$$E=
	U_1\cup_{id} H_{-t-1} \cup_{\tau_{a_t}^{u_t}}\ldots\cup_{\tau_{a_1}^{u_1}} H_{-1},0]\cup_{\tau_{b_k}^{v_k}\dots\tau_{b_1}^{v_1}}H_0 \cup_{\tau_{a'_1}^{u_1}}\ldots\cup_{\tau_{a'_t}^{u_t}}H_t \cup _c U_2
	$$
%\noindent  
with the real structure
$$e=
	\begin{cases}
	\id: U_1\to U_2, \\ % \hspace{3mm} x \mapsto x\\
%	U_2\to U_1,  \hspace{3mm} x \mapsto x\\
	c_M: H_i \to H_{-i-1} \mbox{ (fibered as before)}. %  \hspace{3mm} x \mapsto c_M(x).
	\end{cases}
$$ 
\begin{figure}[h]
\centering
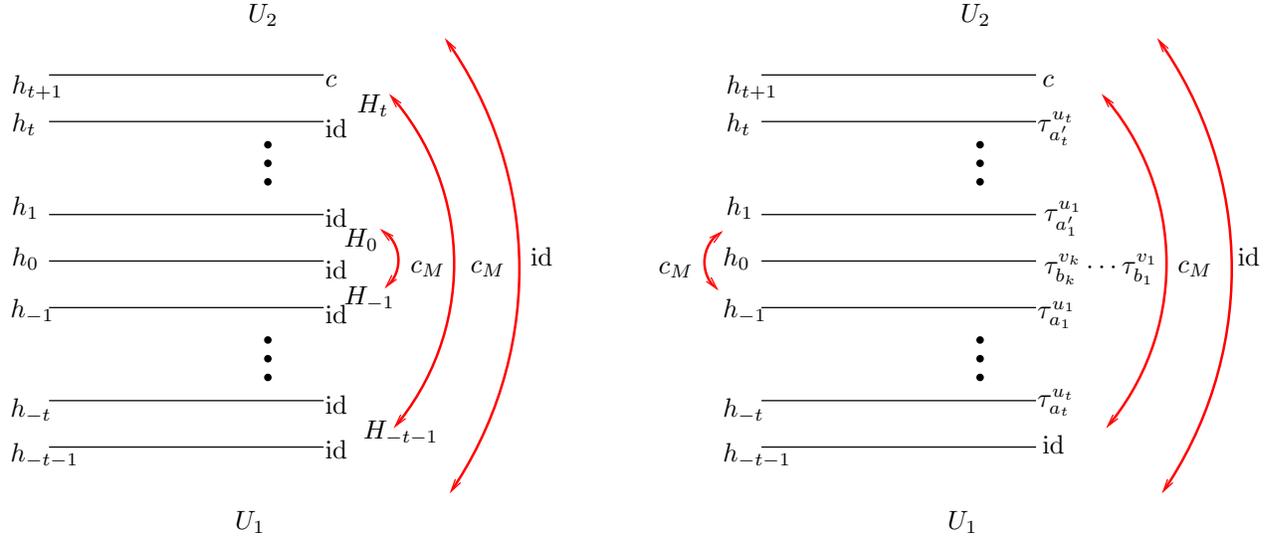
\caption{The real splittings associated to $(M,c_M)$ (left) and $(E,e)$ (right). Gluing maps are shown on the right of the surfaces, real structures are shown as --red-- thick lines.}
\label{model} 
\end{figure}
(See Figure~\ref{model}, right.)
For the sake of simplicity of the demonstration, let us take $k=t=1$ and drop the indices and powers. Then  we have a single invariant knot $b$ lying on the surface $h_0$, and a pair of knots $(a,a')$ swapped by the real structure in both manifolds. 
The general case where there are more curves is handled similarly.
Now, take an equivariant neighborhood (a solid torus) $V_0$ of the curve $b$, by first taking an equivariant annular neighborhood of $b$ on $h_0$, product with a small interval $I$ so that the neighborhood remains equivariant and between but away from the surfaces $h_1$ and $h_{-1}$. Similarly take an equivariant pair of  neighborhoods in $M$ as follows: neighborhood $V'$ of $a'$ between $h_2$ and $h_0$, and $V$ of $a$ between $h_0$ and $h_{-2}$. 
Maybe the only essential observation here is that since $a$ is invariant under $\tau_a$, $\id(V)$ is a neighborhood of $a$ in $ E$ too (similarly for $a'$ and $b$).  Between the complements of the interiors of these solid tori in $M$ and $E$, there is the equivariant diffeomorphism
  $$ \id : M-(\mathring{V'}\cup\mathring{V_0}\cup\mathring{V})\to E -(\mathring{V'}\cup\mathring{V_0}\cup\mathring{V});$$
here $\mathring{V}$ denotes the interior of $V$.

For $V$ in $M$, take a meridian  $\mu$ and a longitude  $\lambda$ that is a copy of $a$; then  $V$ in $E$ has meridian  $\mu\pm\lambda$. Similarly for $V_0$ and $V'$. Then  we excise these solid tori from $M$ and glue back identical solid tori $T,T_0,T'$  respectively. The gluing maps are  $(\pm1)$-sloped boundary diffeomorphisms $\psi,\psi'$ (i.e. $\mu\mapsto \mu\pm\lambda$) for $V$ and $V'$, while for $T_0$
an equivariant diffeomorphism $\psi_0: (\partial T_0,c_*) \to (\partial V_0,e|_{V_0})$ can be taken. Here $c_*=c_1$, $c_2$ or $c_4$. 
Thus one obtains $(M(L),c(L))$ where the real structure $c(L)$ is  given as:
	$$c(L)=
	\begin{cases}
	c_M: M-\mathring{V'}- \mathring{V_0}-\mathring{V} \to M-\mathring{V'}- \mathring{V_0}-\mathring{V},  \\ %\hspace{3mm} x \mapsto c_M(x)\\
	\id: T\to T', \\ %\hspace{3mm}  x \mapsto x\\
	c_*: T_0\to T_0,  \\ %\hspace{3mm}  x \mapsto c_* (x)\\
%	T\to T', \hspace{3mm}   x \mapsto x.
	\end{cases}$$ 
	
Now we define a diffeomorphism $F$ from $M_L$ to $E$ as follows:
$$F=
	\begin{cases}
	\id: M-\mathring{V'}- \mathring{V_0}-\mathring{V} \to E-\mathring{V'}- \mathring{V_0}-\mathring{V},  \\ %\hspace{3mm} x \mapsto x\\
	\psi': T'\to V',  \\%\hspace{3mm}  x \mapsto \psi'(x)\\
	\psi_0: T_0\to V_0, \\ %\hspace{3mm}  x \mapsto \psi_0(x)\\
	\psi: T\to V. %, \hspace{3mm}  x\mapsto \psi(x).
	\end{cases}$$
It is straightforward to check that $F\circ c(L)=e\circ F$ so that $F$ is an equivariant diffeomorphism.

In the second step, we prove that $(E,e)$ is equivariantly diffeomorphic to $(M', c_{M'})$. Consider the following splitting of $(M', c_{M'})$:
	$$U_1\cup_{\id} H_{-t-1} \cup_{\id}\ldots\cup_{\id} H_{-1} \cup_{c'}H_0 \cup_{\id}\ldots\cup_{\id}H_t \cup _{\id} U_2	$$
with the real structure 
	$$c_{M'}=
	\begin{cases}
	\id: U_1\to U_2,  \\ %\hspace{3mm} x \mapsto x\\
	\id: H_i \to H_{-i-1}. %  \hspace{3mm} x \mapsto x.
	\end{cases}$$
\begin{figure}[b]
\centering
%\resizebox{10cm}{!}
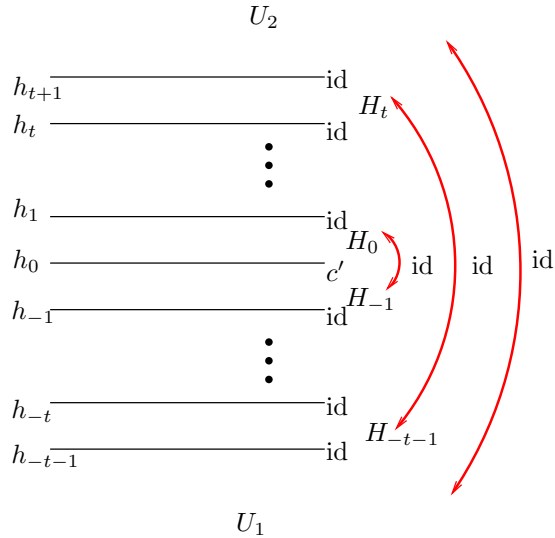
\caption{The splitting associated with $(M', c_{M'})$.}
\label{fig:sfig1} 
\end{figure}
(See Figure~\ref{fig:sfig1}.)	
For each $i$ define the functions $f_i$ and $\overline{f_{i}}$ as follows:
\begin{flalign*}
& f_i = \left(c \cdot \tau_{a_t'}^{-u_t}\dots\tau_{a_{i+1}'}^{-u_{i+1}} \right) \times \id : H_i\to H_i \\
& \overline{f_{i}}=\left(\tau_{a_t}^{u_t}\dots\tau^{u_{i+1}}_{a_{i+1}}\right)\times \id:      H_{-i-1}\to H_{-i-1}
\end{flalign*}
Now define a map from $ (E,e)$ to $(M', c_{M'})$ using the above functions:	
	$$F=
	\begin{cases}
	\id: U_1\to U_1, \\ %  \hspace{3mm} x \mapsto x\\
	\id: U_2\to U_2, \\ %  \hspace{3mm} x \mapsto x\\
	f_i: H_i \to H_i, \\ %  \hspace{3mm} x \mapsto f_i(x)\\
	\overline{f_i}: H_{-i-1} \to H_{-i-1}, % \hspace{3mm} x \mapsto \overline{f_i}(x).
	\end{cases}$$
It is straightforward to check that $F$ is a well-defined diffeomorphism from $E$ to $M'$ on the boundaries of the thickened surfaces and handlebodies. We do that just for $h_0$:
$$
F^{-1}c'F|_{h_0\subset \partial H_0} = \overline{f_0}^{-1}c'f_0
= \tau_{a_1}^{-u_1}\dots\tau_{a_t}^{-u_t} c' c \tau_{a_t'}^{-u_t}\dots\tau_{a_{1}'}^{-u_1}
= \tau_{b_k}^{v_k}\dots\tau_{b_1}^{v_1}.
$$
Finally $F$ is an equivariant diffeomorphism: 
$$
F e  F^{-1}=
\left.	\begin{cases}
	\id: U_1\to U_2 \\
	 \overline{f_i}^{-1} c_M f_i^{-1}=\id: H_i \to H_{-i-1}  %\hspace{3mm} 
	\end{cases} \right\}= c_{M'}
$$
\end{proof}

\noindent	{\em Proof of Theorem~\ref{elw}. }
Let $(M,c_M)$ be a real 3-manifold with  a real Heegaard splitting $(H,c)$ and fix a real Heegaard splitting $(H,s)$ of  $(S^3,\cst)$ with the same surface $H$ ($s$ is to be determined below). The above lemma proves in particular that in case $c \cdot s$ has a 
factorization which is an equivariant product of Dehn twists,  one  gets an equivariant surgery description of a real 3-manifold $M$ starting from $(S^3,\cst)$. So to prove the theorem it suffices  to show that $c  \cdot s$ has a factorization which is an equivariant product.
Indeed first assume that  the real part of $M$ is non-empty and 
consists of the circles $r_0,\ldots,r_k$ in $H$. The real structure $c'=\tau_{r_1}\ldots\tau_{r_k}\cdot c$ on $H$ has the real part $r_0$ and is non-separating if and only if either $k\geq1$  or $k=0$ and $c$ is non-separating. We know that there is a real Heegaard splitting $(H, s)$ of $(S^3,\cst)$ with a non-separating real part $r\subset H$. Furthermore in case $k=0$ but $c$ is separating,  $(S^3,\cst)$ can be assigned a real Heegaard splitting $(H, s)$ with separating $s$ (e.g. see Figure~\ref{separatingob}). By the classification of real structures on compact surfaces, $c'$ and $s$ are conjugate by an orientation preserving diffeomorphism $f$. Set $b_i=f(r_i)$ and note that $b_0=r$. Then we have
		\begin{align*}
s &= f^{-1}\cdot c'\cdot f \\
&= f^{-1}\cdot \tau_{r_1}\ldots\tau_{r_k}\cdot c\cdot f \\
&= 		( f^{-1}\cdot \tau_{r_1}\cdot  f )\cdot f^{-1}\ldots  f\cdot (f^{-1}\cdot \tau_{r_k}\cdot  f)\cdot f^{-1} \cdot c\cdot f \\
&= 		\tau_{b_1}\ldots\tau_{b_k}\cdot f^{-1}\cdot c\cdot f \\
		\end{align*}
		The curves $b_1,\ldots,b_{k}$ on $H$ are now $c_4$-invariant with respect to $s$.		 Let $f=\tau^{u_t}_{a_t}\ldots\tau^{u_1}_{a_1}$ be an expression for $f$ in terms of generators of the mapping class group of $H$. Then 
		\begin{align*}
		c=&f^{-1}\cdot \tau_{b_k}^{-1}\ldots\tau_{b_1}^{-1}\cdot s\cdot f\\
		=&\tau^{u_t}_{a_t}\ldots\tau^{u_1}_{a_1} \cdot\tau_{b_k}^{-1}\ldots\tau_{b_1}^{-1}\cdot s\cdot \tau^{-u_1}_{a_1}\ldots\tau^{-u_t}_{a_t}\\
		=&\tau^{u_t}_{a_t}\ldots\tau^{u_1}_{a_1}  \cdot\tau_{b_k}^{-1}\ldots\tau_{b_1}^{-1}\cdot (s\cdot \tau^{-u_1}_{a_1}\cdot s^{-1})\cdot s\ldots s^{-1}\cdot (s\cdot\tau^{-u_t}_{a_t}\cdot s^{-1})\cdot s\\
		=&\tau^{u_t}_{a_t}\ldots\tau^{u_1}_{a_1} \cdot\tau_{b_k}^{-1}\ldots\tau_{b_1}^{-1}\cdot \tau^{u_1}_{s(a_1)}\ldots\tau^{u_t}_{s(a_t)}\cdot s.
		\end{align*}
The factorization for $c\cdot s$ that appears in the last line above is an equivariant product for $s$.
We complete the proof using the Lemma \ref{main} and its proof to get an equivariant link of the form $\mL= L\cup L_S\cup \overline{L}$ so that $L_S$ is an equivariant unlink consisting of $c_4$-knots and $\cst(L)=\overline{L}$. 
	
For the case where the real manifold $(M,c_M)$ has empty real part, by 
\cite[Lemma 4.3.1]{SaT}, $\tau_r\cdot s$ is a real structure with empty real part. We follow exactly the same steps and end up with an equivariant link $\mL= L\cup L_S\cup \overline{L}$ so that $L_S$ is a $c_2$-knot (the only real knot) in $S^3$ and $\cst(L)=\overline{L}$.
\hfill $\Box$

Let us note the crucial issue that the surgery along the $c_4$-knots in $L_S$ is of type~$4_2$ (purely by the construction of the proof). In other words, the real structure $c_*$ that appears in the expression of the equivariant diffeomorhism $\psi_0$ in the proof of Lemma~\ref{main} is $c_2$.

\section{Three-manifolds with real contact structures}
\label{3MRCS}
The main result of this section is Theorem~\ref{emartinet} that asserts that every real 3-manifold admits a real contact structure. In order to prove that, we simply show that the equivariant surgeries in the previous section can be realized in the contact setting. 
The proof displays a constructive algorithm.
Before that we prove two technical lemmata.

The first one is a Legendrian realization principle in the real setting. Recall that the  Legendrian realization principle Theorem~3.7 in \cite{Ho} gives a criterion for making a collection $C$ of non-isolating curves and arcs on a convex surface Legendrian after a perturbation of the convex surface. (A set $C$ of disjoint curves is said to be non-isolating if $C$ is transverse to $\Gamma_S$, every arc in $C$ begins and ends on $\Gamma_S$, and every component of $S- C$ has nonempty intersection with $\Gamma_S$.)
We are going to prove a weaker version in the real setting, for an invariant set of curves on an anti-symmetric convex surface, imposing an additional condition on the set $C$. 

\begin{lemma}
Let $S$ be an anti-symmetric closed convex surface in $(M, \xi, c_M)$ so that $c=c_M|_S$ is an orientation reversing involution and $\Gamma_S$ is an anti-symmetric oriented dividing set. Let $C$ be a non-isolating $c$-invariant set of curves.
Assume  $S-(\Gamma_S \cup C\cup\Fix~(c))$ consists only of disjoint components $S_1,S'_1,\ldots S_k,S'_k$ with $c(S_i)=S'_i$ for $1\leq i\leq k$. Then there exists an equivariant isotopy $\phi_s$  $(s \in [0, 1])$ of $S$: \\
(1) $\phi_0 = \id$;\\
(2) $\phi_s(S)$ is convex;\\
(3) $\phi_1({\Gamma_S}) = \Gamma_{\phi_1(S)}$;\\
(4) $\phi_1(C)$ is Legendrian and $(\phi_1\circ c\circ \phi_1^{-1})$-invariant.
\label{lerp}
\end{lemma}
\begin{proof}
In the original proof of Legendrian realization principle, one constructs a singular foliation $F$ containing the set $C$ (or an isotopic copy) which is also divided by $\Gamma_S$. Then using the Flexibility Theorem, this foliation can be made a characteristic foliation for $\Gamma$ and since the leaves of the characteristic foliation are Legendrian, so is the set $C$.  For a detailed  proof see \cite{Ho} or \cite{Et}.

We will not reproduce the proof here but instead we will point out how one can modify it equivariantly. The singular foliation $F$ is constructed around a boundary collar of each component of $S-(\Gamma_S \cup C)$ and then is extended to the interior. The assumption that $C$ is non-isolating makes it possible to extend $C$ to a singular foliation. 

In our case, we consider the set $C\cup \Fix~(c)$. $\Gamma_S$ and $\Fix~(c)$ intersect transversally and nontrivially since $\Gamma_S$ is anti-symmetric, which guarantees that $C\cup \Fix~(c)$ is non-isolating provided $C$ is non-isolating. Since $\Fix~(c)$ is Legendrian by definition, the components of  $S-(\Gamma_S \cup C\cup\Fix~(c))$ may be treated as if they are components of $S-(\Gamma_S \cup C)$ with Legendrian boundary $\Fix~(c)$. Now we can construct a singular foliation on each piece $S_i$ in exactly the same way as in the original proof, then extend it to $S'_i$ equivariantly. Then we use the real version of Giroux's Flexibility Theorem \cite[Proposition~3.4]{OzSa1} to conclude.
\end{proof}
Note that if $c$ is separating, then $S-(\Gamma_S \cup C\cup\Fix~(c))$ already consists of equivariant pair of components. The assumption is needed when $c$ is non-separating.

The last assumption of the lemma is indispensable. Indeed we are going to show in the next section that there are cases where the assumption does not hold and  equivariant Legendrian realization is impossible (see Remark~\ref{yok}).

The second lemma is about convex Heegaard splittings of the real tight $S^3$, similar to the observations in \cite[Section~2.1]{OzSa2}.
 
\begin{lemma}
For any genus $g\geq 1$, $(S^3,\xi_{\st},\cst)$ has a real Heegaard splitting where the Heegaard surface $H$ is convex. Moreover, there are $g$ disjoint $c_4$-circles $m_1,\ldots,m_g$ on $H$ such that $|m_i\cap \Gamma|=4$ where $\Gamma $ is the dividing set of $H$.
\label{convex}
\end{lemma}
\begin{proof}
We will start with a real open book decomposition of $(S^3,\cst)$ with disk pages and by positive real stabilizations we will increase the genus of the induced Heegaard splitting. This proves the first claim as the Heegaard surface induced by an open book decomposition is convex.

Now let $(S_0=D^2, f_0=\id,c_0=\rho_0)$ be the real open book decomposition of $(S^3,\cst)$ where $\rho_0$ is the reflection with respect to the $y$-axis. After a real positive stabilization of type~II, the resulting real pages are annuli with the real structure shown in Figure \ref{g1}.

\begin{figure}[h]
\centering
\resizebox{11cm}{!}
{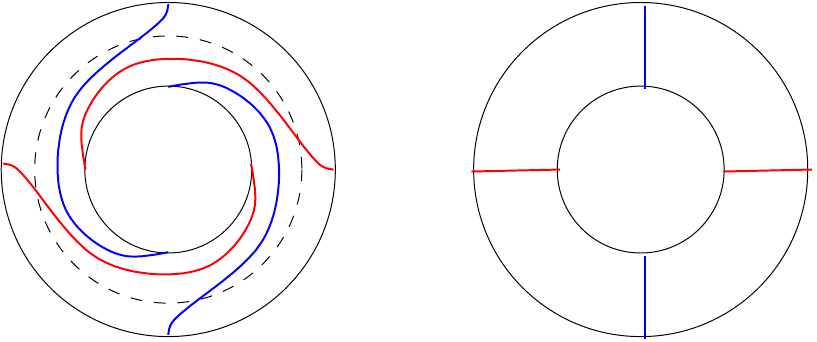}
\caption{The real pages after a real positive stabilization of type~II; left: $(S_-,c_-)$ and right: $(S_+,c_+)$}
\label{g1}
\end{figure}

Performing real positive stabilizations of type~II $g$ times successively, the real pages $(S_-,c_-)$ and $(S_+,c_+)$ become disks with $g$ holes with $c_-=\rho_g \tau_{a_1} \dots \tau_{a_g}$ and $c_+=\rho_g$ where $\rho_g$ is the reflection with $g+1$ real arcs; see  Figure \ref{g+}.

\begin{figure}[h]
\centering
\resizebox{11cm}{!}
{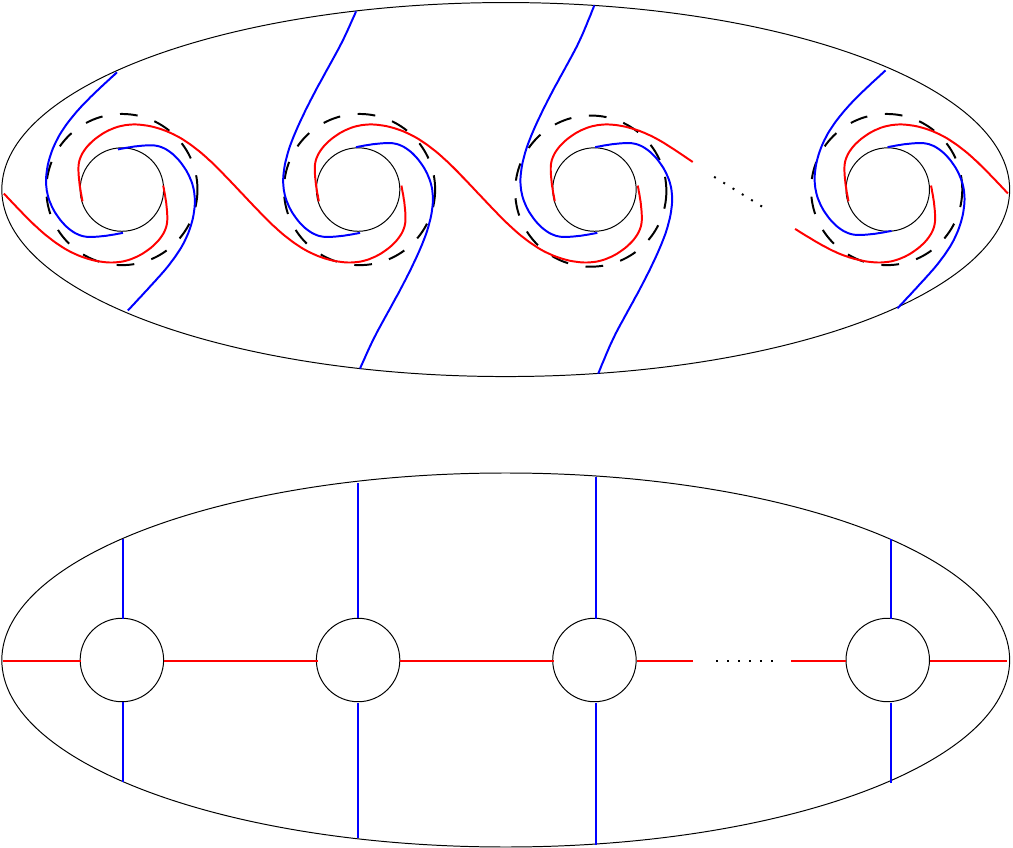}
\caption{The real pages $(S_-,c_-)$ (top) and $(S_+,c_+)$ (bottom) for $(S^3,\cst)$. The red arcs are real, the blue arcs are $\cst$-invariant.}
\label{g+}
\end{figure}

We obtain a real Heegaard splitting with convex surface $H=S_-\cup S_+$ of genus $g$. The dividing set is the binding  and the gluing map is $s= c_- \cup c_+$. The blue curves shown in Figure \ref{g+} are $c_4$-curves with respect to $s$ and satisfy the conditions of the lemma.
\end{proof}

\noindent {\em Proof of Theorem~\ref{emartinet}. }
Let $(M,c_M)$ be a real 3-manifold with the real Heegaard splitting $(H,c)$. By Theorem \ref{elw} there is an equivariant link $\mL=L\cup L_S\cup\overline{L}$ and $\sigma\in\{-1,+1\}^{n}$ in $(S^3,\cst)$ such that $(S^3_{\mL},c({\mL}),\sigma)$ is equivariantly diffeomorphic to $(M,c_M)$. Here the link $L_S$ consists of unlinked $\cst$-invariant knots  and $L$ and $\overline{L}$ are swapped by $\cst$. To turn this equivariant surgery into an equivariant contact surgery, we must put $L_S$ into an equivariant Legendrian position so that the components of $L_S$ satisfy the  hypotheses
of Theorem~\ref{contactsurgery}, summarized in Table~\ref{tab:extensions}. For $L$ and $\overline{L}$, it is enough to make them Legendrian respecting the equivariance.
 
Using the set up  in the proof of Theorem \ref{elw}, let $\Fix~(c)$ consist of $k+1$ real circles ${r_0,\ldots,r_k}$. Consider $(S^3,\cst)$ with the real contact Heegaard splitting $(H,s)$ of genus $k$ induced by the real open book decomposition given in Lemma \ref{convex}. Let $r\subset H$ denote the real circle of $s$. There are $k$ disjoint $c_4$-circles (with respect to $s$) ${m_1,\ldots m_k}$ on $H$ by Lemma~\ref{convex}. Then $c'=\tau_{m_1}^{-1}\ldots\tau_{m_k}^{-1}\cdot s$ is a real structure with $\Fix~(c') =\{r, m_1,\ldots, m_k\}$. Now that $c$ and $c'$ are conjugate, there is an orientation preserving diffeomorphism $f$ of $H$ given by a product $\tau^{\sigma_1}_{a_1}\ldots\tau^{\sigma_t}_{a_t}$ such that $c=f\circ c'\circ f^{-1}$. Following the steps in the proof of Lemma \ref{main}, we get an equivariant surgery link $L\cup L_S\cup\overline{L}$ where $L_S$ consists of the knots ${m_1,\ldots, m_k}$ on $H$. On each $m_i$, we perform a type-$4_2$, equivariant $(-1)$-surgery. Note that since $\{m_1,\ldots, m_k\}$ is a non-isolating set of curves transverse to the dividing set, we can make each $m_i$ Legendrian  with an equivariant perturbation of $H$ by Lemma \ref{lerp}. Moreover for each modified $m_i$, $\tw=-2$ since $|m_i\cap \Gamma|=4$. Hence, the condition in Theorem~\ref{contactsurgery} to perform an equivariant contact surgery on a $c_4$-knot is satisfied and we perform an equivariant $(+1)$-contact surgery on each $m_i$. 
 
Note that if $k=0$ and $r_0$ is separating, we cannot use the real Heegaard splitting given by Lemma~\ref{convex} since $r$ is non-separating there, so that $c$ and $s$ would not be conjugate. In that case we use the real Heegaard splitting induced by the real open book given in Figure~\ref{separatingob}. (See also \cite[Figure~6]{OzSa2}.) This open book is obtained  from the simplest real open book decomposition of $S^3$  with disk pages and succesively  attacing handles of type~III  $k$ times. The real circle is separating in this case and $c$ and $s$ are conjugate. The rest of the proof follows similarly as above.
\hfill $\Box$
\begin{figure}[h]
\centering
\resizebox{11cm}{!}
{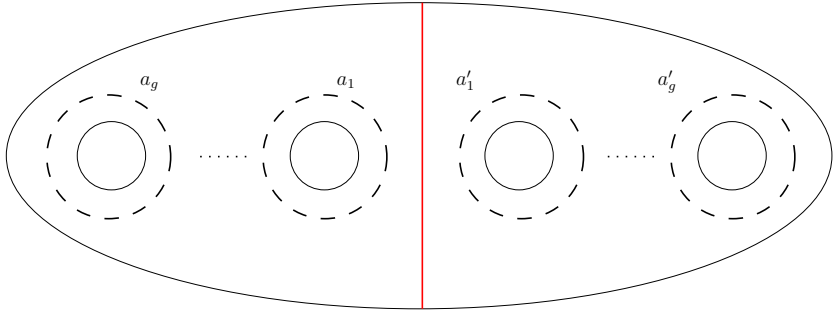}
\caption{ A real open book of $S^3$ after stabilizations of type~III. The real page $(S_-,c_-)$ is depicted.}
\label{separatingob}
\end{figure}

\noindent {\em Proof of Corollary~\ref{OTinZhomo}. }
This follows immediately from Proposition~\ref{OTinS3} and the proof of Theorem~\ref{emartinet}.
Given any real integer homology sphere $(\Sigma,s)$ the proof of Theorem~\ref{emartinet} shows that one can obtain an $s$-real contact structure $\eta$ on $\Sigma$ via an equivariant contact diagram $\Gamma$ in $(S^3,\cst,\xi_{\st})$.  
Then consider the real contact 3-manifold $(\Sigma,s,\eta)\#(S^3,\cst,\xi_{p}), p\in\Z$, which is equivariantly diffeomorphic to $(\Sigma,s)$. Besides the final contact structure is overtwisted, since $\xi_p$ is. In an integer homology sphere, the set of  overtwisted contact structures up to contact isotopy  is countably infinite and they are distinguished by the $d_3$~invariant.  Again from \cite[Lemma~4.2]{DiGeSt} it follows that varying $p$ one exhausts all overtwisted structures on $(\Sigma,s)$. Each is given 
by the equivariant contact diagram $\Gamma$ adjoined with the one for $(S^3,\cst,\xi_{p})$ obtained in the proof of Proposition~\ref{OTinS3}.
\hfill $\Box$

\section{Examples}
\label{EXs}
This section is a display for the applications of  the ideas we have harvested in the previous sections. For some basic real 3-manifolds ($\s$ and lens spaces), we will give equivariant surgery descriptions and provide equivariant contact surgery diagrams in $(S^3,\cst,\xi_{\st})$. Of course, given surgery diagrams for these manifolds we could readily make them equivariant. However then we would not have any control on the real structure obtained at the end. In this section we turn surgery diagrams into equivariant ones with a careful control on the final real structure. For a collection of similar results see \cite{OnOz}.

Below we use the genus-1 real contact Heegaard decomposition  of  $(S^3,\cst)$ with the gluing map $\cst=\begin{bmatrix} 0& 1\\1 & 0\end{bmatrix}$, with abuse of notation.  The convexity of the  Heegaard surface is guaranteed by
\begin{lemma}
\label{cxTorusinS3}
(a) There is a real contact Heegaard decomposition of $(S^3,\cst,\xi_{\st})$ 
of genus~1, with the Heegaard map $\cst$. \\
(b) Moreover a parallel copy of the dividing set (which does not satisfy the hypothesis of Lemma~\ref{lerp}) can be realized in a Legendrian way on the Heegaard surface. 
\end{lemma}
\begin{proof}
We consider the setup in \cite[Example 4.6.21]{Ge} in the real  3-sphere $\{r_1^2+r_2^2=2\}\in\R^4$ with
$$\cst(r_1,\varphi_1,r_2,\varphi_2)=(r_2,-\varphi_2,r_1,-\varphi_1).$$  
The invariant torus $T:r_1=r_2=1$ of the Hopf vector field is an embedded anti-symmetric Heegaard surface naturally, with the gluing map $\cst$. As noted in \cite{Ge},  such a torus is not convex. However it can be made convex equivariantly using the  equivariant perturbation similar to  \cite[Example 4.8.10]{Ge}. There the ambient contact space $(T^2\times \R, d\theta/2+z d\varphi)$ can be made real by defining $c'(\theta,\varphi,z)=(-\theta,\varphi,-z)$. Now an equivariant neighborhood of the anti-symmetric torus $\{z=0\}$ is equivariantly contactomorphic to an equivariant neighborhood of $T$ via the map 
$f:\theta=\varphi_1+\varphi_2,\varphi=\varphi_1-\varphi_2,z=r_1^2-1$, i.e. $f$ satisfies $f\circ c_{\st}=c' \circ f$.
Finally observe that the second characteristic foliation  in \cite[Example 4.8.10]{Ge} is $c'$-real when $p=1$ and $q=0$. In that (standard) characteristic foliation a knot parallel to the dividing set is Legendrian (see the corresponding figure in \cite{Ge}). 
\end{proof}

Now we are ready to move on towards our basic examples.

\subsection{$\s$}

We refer to the work of J.~L.~Tollefson \cite{To} for the classification of involutions on $\s$ up to conjugation by a diffeomorphism of $\s$. There are 13 classes in total and 4 of them are orientation preserving with one dimensional fixed point set. 
Hence there are 4 real structures on $\s$ with nonempty real part, up to equivalence. 
If we identify  $\s$  as $\{(\theta,(x,y,z)): x^2+y^2+z^2=1\}\subset S^1\times\R^3$ then we can express two of these real structures: \\
\noindent (1) $s_1:(\theta,(x,y,z))\mapsto (\theta,(-x,-y,z))$,\\
\noindent(2) $s_2:(\theta,(x,y,z))\mapsto (-\theta,(x,y,-z))$,\\
each of which has two real circles.

If we  consider $\s$ as  the space $[-1, 1] \times S^2$ with $(-1, (x,y,z))$ and $(1, (-x,-y,z))$ identified, then we can express the remaining two of these real structures as: \\
\noindent(3) $s_3:(t,(x,y,z))\mapsto (t,(x,-y,-z))$,\\
(4) $s_4:(t,(x,y,z))\mapsto \begin{cases}
(1-t,(-x,-y,-z)) \text{ if }  0\leq t\leq 1\\
(-1-t,(x,y,-z)) \text { if }  -1\leq t < 1
\end{cases}$,\\
each of which has one real circle.

Each of the real structures above can be associated with a real genus-1 Heegaard  splitting of $\s.$
Consider the real structures on $T=S^1\times S^1$ given by the matrices
\begin{equation}
s_1=\begin{bmatrix} -1 & 2\\0 & 1\end{bmatrix}, s_2=\begin{bmatrix} 1 & 2\\0 &-1\end{bmatrix}, s_3=\begin{bmatrix} -1 & 1\\0 & 1\end{bmatrix}, s_4=\begin{bmatrix} 1 &1\\0&-1\end{bmatrix}.
\label{s1234}
\end{equation}
The abuse of notation here relies on the fact that the 
abstract real Heegaard splitting $(T,s_i)$ for any $k\in\Z$  is associated to the real 3-manifold $(\s,s_i)$.  This can be seen by computing the quotients and the number of real components. (Any even -resp. odd- integer at the top right entry of $s_1$ or $s_2$ -resp. $s_3$ or $s_4$- would give the same result.)

Now that we have an abstract real Heegaard splitting for each real $\s$, we will demonstrate how to obtain them via equivariant surgery from the standard real 3-sphere via Theorem \ref{elw}. Below $a$ and $b$ are the curves $S^1\times \{1\}$ and  $\{1\}\times S^1$ on $T$; $\tau_a=\begin{bmatrix} 1& 1\\0 & 1\end{bmatrix}$, $\tau_b=\begin{bmatrix} 1& 0\\-1 & 1\end{bmatrix}$   are the Dehn twists along $a$ and $b$ respectively. 

\noindent (1)  The gluing involution $s_1$  satisfies $s_1=\tau^{-1}_{a+b} \cst$ where $a+b$ is a $c_4$-knot for $\cst$. Then $(\s,s_1)$ can be obtained from $(S^3,\cst)$ by an equivariant $(-1)$-surgery of type~$4_2$ on the unknot represented by $a+b$, with respect to the framing induced from the Heegaard surface $H$. 
Since the Heegaard framing minus the Seifert framing  along $a+b$  is $+1$, a $(-1)$-surgery with respect to $H$ corresponds to a topological $0$-surgery on the unknot, which gives of course $\s$.

\noindent (2)  We have $s_2=\tau_{a-b} \cst$ where $a-b$ is a $c_1$-knot for $\cst$. Then $(\s,s_2)$ can be obtained from $(S^3,\cst)$ by an equivariant $(+1)$-surgery of type~$1_1$ on the unknot represented by $a-b$, with framing with respect to  $H$. 

\noindent (3)    For $(\s,s_3)$, we have $s_3=\tau_{b}^{-1}\tau_{a}^{-1}\cst$, and $\cst(a)=b$. Hence $(\s,s_3)$ is given by an equivariant $(-1)$-surgery of type~5 on the Hopf link formed by the knot pair $(a,b)$. Here Seifert framing and Heegaard framing of $a$  coincide; similarly for $b$. 

\noindent (4)   For $(\s,s_4)$, we have $s_3=\tau_{b}\tau_{a}\cst$. Hence $(\s,s_3)$ is given by an equivariant $(+1)$-surgery of type~5 on the Hopf link formed by the knot pair $(a,b)$.

The tight contact structure on $\s$, unique up to isotopy, is given by $$\xi = ker(x d\theta + y dz - z dy).$$ Obviously the contact structure $\xi$ is real with respect to $s_1$ and $s_2$. However it is not obvious at all whether $\xi$ is real with respect to $s_3$ and $s_4$. Below we also answer that question for $s_4$ affirmatively using equivariant contact surgery.

The unique tight contact structure $\xi$ of $\s$ is obtained from $(S^3,\xi_\st)$ through contact $(+1)$-surgery on the unknot with Thurston-Bennequin number $\tb=-1$ \cite{DiGeSt} (see Figure~\ref{s1xs2_12}, left). 
\begin{figure}[h]
\centering
\resizebox{14cm}{!}
{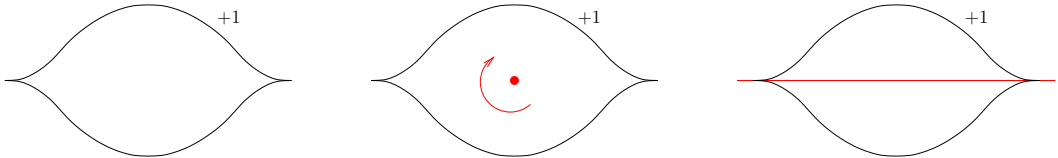}
	\caption{The contact surgery diagram (left) for the unique tight contact structure on $\s$ can be made $s_1$-real (middle) and $s_2$-real (right)}
	\label{s1xs2_12}
\end{figure}
This contact surgery diagram can be readily made equivariant in two ways as depicted  in Figure~\ref{s1xs2_12}, middle and  right. Now we argue that these correspond to the equivariant contact surgery descriptions for $(\s,s_1,\xi)$ and  $(\s,s_2,\xi)$ respectively.

We have already observed that $(\s,s_1)$ is given  by an equivariant $(-1)$-surgery (with respect to $H$) of type~$4_2$ on the unknot represented by $a+b$.  
We consider a real Heegaard splitting $(H,\cst)$ of $(S^3,\cst,\xi_{\st})$; here, thanks to Lemma~\ref{cxTorusinS3}(a), we may assume that   $H$ is an equivariant convex torus with two parallel dividing curves, each represented by $a-b$. 
Here, the  Real Legendrian Realization Principle Lemma~\ref{lerp} cannot be used to get  $a+b$  equivariantly Legendrian on $H$. Instead we employ Lemma~\ref{cxTorusinS3}(b) to accomplish that.
Since the contact twisting along $a+b$ with respect to $H$ is $\tw(a+b;H)=-2$,  
the contact surgery coefficient is $+1$. Moreover, $\tb(a+b)=-1$.
Hence we get the corresponding equivariant contact surgery diagram, the one in the middle in Figure \ref{s1xs2_12}.	

Similarly $(\s,s_2)$ is given by an equivariant $(+1)$-surgery of type~$1_1$  (with respect to $H$) on the unknot represented by $a-b$.  In this the curve $a-b$ can be made  equivariantly Legendrian on $H$ according to Lemma~\ref{lerp}. 
Then  $\tw(a-b;H)=0$ as $(a-b)$ is parallel to the dividing set. Hence we get the 
contact surgery coefficient $+1-0=+1$. Since  $\tb(a-b)=-1$, we obtain the  equivariant contact surgery diagram in Figure \ref{s1xs2_12}, right.

$(\s,s_4)$ is obtained  by performing a type~5 $(+1)$-surgery on the knot pair formed by the curves $(a,b)$.
We make $a$ Legendrian (no need to employ Real Legendrian Realization Principle), and an equivariant copy represents a Legendrian unknot $b$. 
If these Legendrian unknots are arranged to have $\tw=\tb=-1$ then we would need an equivariant pair of contact $(+2)$-surgeries.
However such a contact surgery is not well-defined; a priori it is not unique. However in this case it is: either of the two choices made  
during the surgery leads to isomorphic structures  (see \cite[Example~5.3]{Ke}). Anyhow this topological equivariant  surgery must be turned into a well-defined 
equivariant contact one.  For this we follow the idea in \cite{Ke} to turn the contact $(+2)$-surgery along the Legendrian knot with $\tb=-1$ into 
contact $(\pm 1)$-surgeries along a pair of Legendrian knots $K_1, K_2$ and their equivariant copy $L_1, L_2$  as appears in Figure \ref{hopf1}(a).
Moreover, this contact diagram gives the unique tight structure on $\s$. 

\begin{figure}[h]
\centering
\resizebox{14cm}{!}
{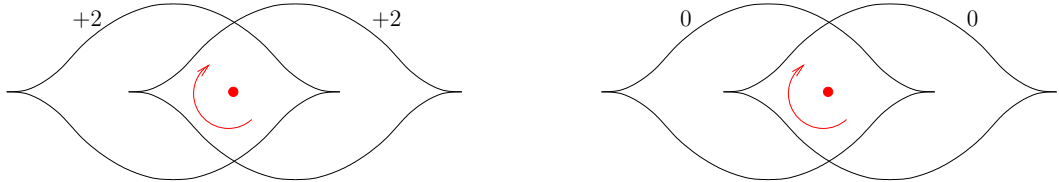}
\caption{Equivariant contact surgery diagrams (a) for $(\s,s_4,\xi)$ and  (b) for  $(\s,s_3,OT)$.}
	\label{hopf1}
\end{figure}

A proof of tightness is shown in Figure \ref{hopf1}.
The first isomorphism is by contact sliding of $K_1,L_2$ over $K_2$ \cite{DiGe}. In the new diagram $L_1$ does not link with the slided copies of $K_1,L_2$ and thus it is a
meridian of $K_2$. The second isomorphism follows from the fact that a $(-1)$-contact surgery on some $L$ and a $(+1)$-contact surgery on its meridian cancels each other.
One way to see that is by sliding the meridian over $L$ to get a Legendrian push-off of $L$. A  $(-1)$-contact surgery on $L$ and a $(+1)$ on its push-off cancel.
The third isomorphism is via applying \cite[Lemma~2.9]{EtKeOn} (for two knots and for $n=+1$). This lemma is a generalization of \cite[Proposition~2.4]{LiSt}.
Finally the last isomorphism is a contact Reidemeister move.
The discussion up to now proves Theorem~\ref{3ofs1xs2}.

Furthermore $(\s,s_3)$ is given by a type~5 $(-1)$-surgery on the same Hopf link as above and a possible corresponding equivariant contact surgery diagram is shown in Figure \ref{hopf1}(b). This $s_3$-real contact structure is overtwisted (see e.g. \cite{DiGeSt}). So the attempt to find out whether $\xi$ is $s_3$-real or not is inconclusive with this particular surgery. We still do not know whether $\xi$ is $s_3$-real or not.

\begin{remark}
\label{yok}
Let us start with the genus-1 real contact Heegaard decompositions  of  $(S^3,\cst)$ with the gluing map   $f=\tau_b s'_2$ instead of $\cst$, where $s'_2=\begin{bmatrix} 1 & 0\\0 &-1\end{bmatrix}$. This decomposition is induced from the real open book depicted in Figure~\ref{g1}, thus the Heegaard surface $H$ is equivariantly convex as it comes from a real open book.
Since $s'_2=\tau^{-1}_{b} f$,  $(\s,s_2)$ can be obtained from $(S^3,\cst)$ via an equivariant $(-1)$-surgery of type~$1_1$ on the unknot represented by $b$, with framing with respect to  $H$. 
We observe that $b$ can be realized in a Legendrian way on $H$ by the Legendrian Realization Principle but that cannot be done equivariantly. Indeed, first note that Lemma~\ref{lerp} does not apply here as its last assumption is not satisfied. If equivariant realization were possible, then  $\tw(b;H)$ would be zero as $b$ is parallel to the dividing set. Hence we would get $\tb(b)=0+1=1$. This is of course impossible for an unknot in the tight $S^3$.
\end{remark}

\subsection{Lens spaces -- real structures of type~A}
\label{lenstypeA}

As in the previous paragraph let us consider an $h$-equivariant Heegaard torus $T$ in the real lens space $(L(p,q),h)$, bounding solid tori $V$ and $V'$. If $h$ is orientation preserving on $T$  and if $h|_{V}=c_1$ and $h|_{V'}=c_1$, then $h$ is said to be of type~$A$. It is interesting that for any $p,q$, $L(p,q)$ admits a real structure of type~$A$, unique up to equivariant isotopy \cite{HoRu}. Moreover we claim and prove here Theorem~\ref{A-real} which asserts that any tight contact structure on $L(p, q)$ is $A$-real.
 
Let $p>q>0$ and $ [r_1,\ldots ,r_n]$,  $r_i \leq-2$ be the continued fraction for $-p/q$.  K.~Honda proved that on $L(p,q)$, there exists exactly $|(r_1 + 1) \ldots (r_k + 1)| $ tight contact structures up to isotopy. Moreover, each tight contact structure on $L(p, q)$ can be obtained from Legendrian ($(-1)$-contact) surgery on a link in the standard tight 3-sphere \cite{Ho}. Topologically, surgery link is given by a chain of unknots with (ordered)  coefficients $r_i$ as in  Figure \ref{chain}. This surgery diagram in fact is an equivariant surgery diagram where all the unknots are $c_1$-knots. Performing equivariant surgery produces $(L(p,q),A)$. Indeed, the quotient of the surgered manifold with the real structure is  $S^3$. Since the only real structure on $L(p,q)$ having quotient $S^3$ is $A$ \cite{HoRu}, the real structure on the surgered manifold is of type~$A$.
\begin{figure}[h]%               	
	\centering
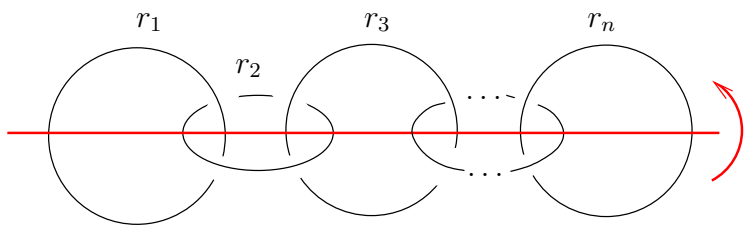
\caption{A surgery diagram for $L(p,q)$, which may  be taken equivariant.}
	\label{chain}
\end{figure}

There are exactly $|(r_i + 1)|$ ways to stabilize an unknot to perform a contact $(-1)$-surgery on it. Moreover, each choice of stabilization  will represent a Legendrian $c_1$-unknot. Hence, each contact surgery diagram describing a tight contact structure on $L(p,q )$ is also an equivariant contact surgery diagram (consisting of Legendrian $c_1$-knots and contact $(-1)$-surgeries of type~$1_1$) describing an $A$-real tight contact structure. 

This argument proves Theorem~\ref{A-real}.

Finally note that the  equivariant contact surgery diagram we mentioned in the previous paragraph cannot be obtained by the construction in the proof of Theorem~\ref{elw} since the latter does not involve any $c_1$-knots.

\end{document}